\documentclass[10pt]{amsart}

\usepackage{amsfonts,amsmath,amsthm,amssymb,amscd,latexsym,enumitem,etoolbox,mathrsfs,comment, graphicx}
\usepackage[all]{xy}
\usepackage{thmtools, thm-restate}
 \usepackage[normalem]{ulem}

\newcommand{\R}{\mathbb{R}}
 
\newcommand{\N}{\mathbb{N}}

\newcommand{\Z}{{\mathbb Z}}

\renewcommand{\d}{\mathrm{d}}

\renewcommand{\phi}{\varphi}

\theoremstyle{plain}
\newtheorem*{theorem*}{Theorem}
    \newtheorem{theorem}{Theorem}[section]
    \newtheorem{lemma}[theorem]{Lemma}
    
    \newtheorem{proposition}[theorem]{Proposition}
    
\theoremstyle{definition}
    \newtheorem{definition}[theorem]{Definition}
    \newtheorem{example}[theorem]{Example}

    \newtheorem{remark}[theorem]{Remark}

\theoremstyle{remark}

\DeclareMathOperator{\id}{id}

\DeclareMathOperator{\homeo}{Homeo}
\DeclareMathOperator{\diff}{Diff}

\title[Minimal homeomorphisms and $K$-theory]{Minimal homeomorphisms and topological $K$-theory}

\author[R.J. Deeley, I.F. Putnam, K.R. Strung]
{Robin J. Deeley \and
Ian F. Putnam \and
Karen R. Strung}
\address{R. J. Deeley, Department of Mathematics,
University of Colorado Boulder,
Campus Box 395 2300 Colorado Avenue,
Boulder, CO 80309-0395, USA }
\email{robin.deeley@gmail.com}

\address{I. F. Putnam, Department of Mathematics and Statistics,
University of Victoria,
Victoria, B.C., Canada V8W 3R4} 
\email{ifputnam@uvic.ca}

\address{K. R. Strung, Institute of Mathematics, Czech Academy of Sciences, \v{Z}itn\'a 25, 115 67 Prague, Czech Republic}
\email{strung@math.cas.cz}

\date{\today}
\thanks{RJD is currently funded by NSF Grant DMS 2000057 and was previously funded by Simons Foundation Collaboration Grant for Mathematicians number 638449. KRS is currently funded by GA\v{C}R project 20-17488Y and \mbox{RVO: 67985840} and part of this work was carried out while funded by Sonata 9 NCN grant 2015/17/D/ST1/02529 and a Radboud Excellence Initiative Postdoctoral Fellowship. IFP is supported in part by an NSERC Discovery Grant.}

\begin{document}

\begin{abstract} 
The Lefschetz fixed point theorem provides a powerful obstruction to the existence of minimal homeomorphisms on well-behaved spaces such as finite CW-complexes. We show that these obstructions do not hold for more general spaces. Minimal homeomorphisms are constructed on compact connected metric spaces with any prescribed finitely generated $K$-theory or cohomology. In particular, although a non-zero Euler characteristic obstructs the existence of a minimal homeomorphism on a finite CW-complex, this is not the case on a compact metric space. We also allow for some control of the map on $K$-theory and cohomology induced from these minimal homeomorphisms. This allows for the construction of many minimal homeomorphisms that are not homotopic to the identity. Applications to $C^*$-algebras will be discussed in another paper.    
\end{abstract}
\maketitle

\setcounter{section}{-1}

\section{Introduction}
Let $X$ be an infinite compact metric space. A homeomorphism $\varphi: X \rightarrow X$ is called \emph{minimal} if it has the following property: if $F\subseteq X$ is a closed non-empty set such that $\varphi(F) = F$, then $F=X$. A fundamental question in the theory of dynamical systems is the following: \\

{\bf Question:} Given $X$, does there exist a minimal homeomorphism $\varphi : X\rightarrow X$? \\

Examples where the answer is positive include the Cantor set, the circle, the torus, the Klein bottle, any odd dimensional sphere, among others. After proving the existence of a minimal homeomorphism on a particular space, one would like to classify all minimal homeomorphisms on that space up to various natural equivalences. 

However, proving existence is typically non-trivial and there are in fact many obstructions to the existence of a minimal homeomorphism on well-behaved spaces. For example, a compact manifold with non-empty boundary cannot admit a minimal homeomorphism. More subtly, if we are given a finite CW-complex with non-zero Euler characteristic, then by a result of Fuller \cite[Theorem 2]{Ful:PerPts} any homeomorphism on it has a periodic point. Since we are considering infinite metric spaces, a homeomorphism with periodic points can never be minimal.

Based on Fuller's result, it is natural to ask if the Euler characteristic is still an obstruction to the existence of a minimal homeomorphism when considering a general compact metric space, rather than a finite CW-complex. This is not the case. In fact, it follows from our results here that given an integer $n$ there exists an infinite compact metric space with Euler characteristic $n$ that admits a minimal homeomorphism. 

In fact, what we prove is more general and one of our main motivations is given by the prominent role minimal homeomorphisms play in the theory of $C^*$-algebras via the crossed product construction. In particular, $K$-theory is an important invariant for both spaces and $C^*$-algebras. For a compact metric space $X$, its topological $K$-theory is a $\Z/2\Z$-graded abelian group denoted by $K^*(X)$. Topological $K$-theory is a functor, so any continuous map $\varphi: X \rightarrow Y$ induces a group homomorphism $K^*(Y) \rightarrow K^*(X)$, which we denote by $\varphi^*$. Furthermore, the $K^0$-group of a compact space has the form $K^0(X)\cong \Z \oplus \tilde{K}^0(X)$ where $\tilde{K}^0(X)$ is reduced $K$-theory, and if $X$ is connected then the group homomorphism $\varphi^*$ acts as the identity on the copy of $\Z$.

If $(X, \varphi)$ is a minimal dynamical system, then, from the $C^*$-algebra perspective, computing $K^*(X)$ and $\varphi^*: K^*(X) \rightarrow K^*(X)$ are  fundamental problems, as they are necessary inputs for calculating the $C^*$-algebra $K$-theory of the associated crossed product $C^*$-algebra, $C(X) \rtimes_{\varphi} \mathbb{Z}$. With our assumptions on $X$, such a \mbox{$C^*$-algebra} is simple, separable, unital and nuclear and an important open problem is determining the $K$-theoretic range of such crossed products. For more about the $C^*$-algebraic applications, see \cite{DPS:minDSKth}.

As many spaces do not admit minimal homeomorphisms, we instead tackle the question from the perspective of $K$-theory and cohomology. That is, rather than trying to determine whether a given space admits a homeomorphism, we fix $K$-theory (or cohomology) and then seek spaces admitting minimal homeomorphisms which realize the prescribed $K$-theory (or cohomology). Our results begin with the precise formulation of this question: 
 
{\bf Question A:} Given countable abelian groups $G_0$ and $G_1$, does there exist an infinite connected compact metric space $X$ such that
\begin{enumerate}
\item $X$ admits a minimal homeomorphism,
\item $K^0(X) \cong \Z \oplus G_0$ and $K^1(X) \cong G_1$?
\end{enumerate}  

Our main existence result is that there is such a space when the groups are finitely generated, see Theorem \ref{Thm:ExistenceOfMinHomeo} for further details. The proof of this result uses an important result of Glasner and Weiss \cite[Theorem 1]{GlaWei:MinSkePro} (see Theorem \ref{GlasnerWeissToZ}  for the special case of this result that we use), our previous work in \cite{DPS:DynZ}, and some Hilbert cube manifold theory. Our solution to this special case of Question A also leads to a positive answer to the analogous question for the existence of minimal homeomorphisms on spaces with prescribed finitely generated \v{C}ech cohomology.

With the existence of a minimal homeomorphism proved for spaces with prescribed finitely generated $K$-theory, we move to understanding the collection of minimal homeomorphisms on these spaces. To do so, we studying the collection of minimal homeomorphisms on a particular space via the maps on $K$-theory and \v{C}ech cohomology they induce. It is important to note that our results in this direction apply not only to the spaces we construct but also to more well-behaved spaces such as manifolds. The starting point is the following result of Fathi--Herman \cite[Th\'eor\`eme 1]{FatHer:Diffeo}:
\begin{theorem}
Suppose that $M$ is a smooth closed connected manifold that admits a free smooth $S^1$-action. Then $M$ admits a minimal homeomorphism that is homotopic to the identity.
\end{theorem}
This result gives a powerful way to obtain existence results. However, any minimal homeomorphism obtained via this theorem induces the identity map on $K$-theory and cohomology because it is homotopic to the identity. We extend their result to allow for more general induced maps.  A special case (see Remark \ref{RemMainTheoremNonManifoldCase}) is the following:
\begin{theorem}[Special case of Theorem \ref{MainTheoremNonManifoldCase}]
Suppose $M$ is a smooth closed connected manifold that admits a free smooth $S^1$-action, $Y$ is a closed connected manifold and $\beta: Y \rightarrow Y$ is a finite order homeomorphism. Then there exists a minimal homeomorphism on $M \times Y$ that is homotopic to $\id_M \times \beta$.
\end{theorem}
In particular, if $\beta$ in the statement of this theorem acts non-trivially on $K$-theory or \v{C}ech cohomology then the minimal homeomorphism also acts non-trivially and hence is \emph{not} homotopic to identity. In addition, if both $Y$ and $\beta$ are smooth then the resulting minimal homeomorphism can also be taken to be a diffeomorphism, see Theorem \ref{mainTheoremManifoldCase} and Example \ref{manifoldGeneralProductExample}. The assumption that $Y$ is a closed connected manifold can be weakened, see Theorem \ref{MainTheoremNonManifoldCase}. This is important as the spaces we construct to answer Question A are not manifolds (and by Fuller's result cannot be manifolds or even CW-complexes).

By including the action on the given abelian groups into Question A, we have the following:  

{\bf Question B:} Given countable abelian groups $G_0$ and $G_1$ and group automorphisms $\sigma_0: G_0 \rightarrow G_0$ and $\sigma_1: G_1 \rightarrow G_1$ does there exist an infinite compact metric space $X$ such that
\begin{enumerate}
\item $X$ admits a minimal homeomorphism $\alpha$,
\item $K^0(X) \cong \Z \oplus G_0$ and $K^1(X) \cong G_1$, 
\item $\alpha^*= \id \oplus \sigma_0$ on $K^0(X)$ and $\alpha^*=\sigma_1$ on $K^1(X)$?
\end{enumerate}  
The spaces constructed in our answer to Question A already allow for a positive answer to Question B when the groups are finitely generated and both $\sigma_0$ and $\sigma_1$ are the identity. However, we are also able to construct systems where the maps  $\sigma_0$ and $\sigma_1$ are not necessarily the identity, allowing even further progress on this question. A special case of our results gives the following (see Section \ref{PointLikeSpacesWithBeta} for details):
\begin{theorem}[see Theorem \ref{SummaryTheorem}]
Suppose that $Y$ is a connected finite CW-complex and $\beta: Y \rightarrow Y$ is a finite order homeomorphism. Then there exist a connected compact metric space $X$ and minimal homeomorphism $\alpha: X\rightarrow X$ such that
\begin{enumerate}
\item $K^*(X) \cong K^*(Y)$ and $\alpha^* = \beta^*$ on $K$-theory, 
\item $H^*(X) \cong H^*(Y)$ and $\alpha^* = \beta^*$ on \v{C}ech cohomology.
\end{enumerate}
\end{theorem}
In particular, this allows us to determine the $K$-theoretical range of the crossed product ${C}^*$-algebras associated to minimal dynamical systems $(X, \varphi)$, when $X$ has finitely generated $K$-theory, see \cite{DPS:minDSKth}. 

\subsection*{Acknowledgements} The authors thank the referee for a careful reading of the paper and many useful comments. In particular, we thank them for finding an error in the original formulation of Theorem \ref{homotopyLemmaNonMfldCase}.

\section{Background}

\subsection{Group actions and minimal homeomorphisms}
There will be a number of group actions considered in the present paper. Our main results centre on $\Z$-actions. A $\Z$-action on a compact metric space $X$ is obtained by iterates of a homeomorphism $\varphi : X \to X$. If the given homeomorphism has finite order then it is often more natural to consider the associated $\Z/n\Z$-action where $n$ is the order of the homeomorphism. We will also need to consider smooth actions of the circle on smooth closed manifolds. The reader should note that all actions of $S^1$ on smooth manifolds in the present paper are smooth actions.

\begin{definition} \label{defLocFree}
If $M$ is a closed manifold then a (smooth) $S^1$-action on $M$ is \emph{free} if the following condition holds: if there exists $\theta \in S^1$ such that $\theta \cdot m =m$ for some $m\in M$, then $\theta=\id_{S^1}$. A (smooth) $S^1$-action is \emph{locally free} if for each $m\in M$, the set
\[
\{ \theta \in S^1 \mid \theta \cdot m = m \}
\]
is finite.
\end{definition}
\begin{definition}
Let $X$ be a compact metric space, and $\Gamma$ a topological group. An action $\Gamma \to \homeo(X) : g \mapsto (x \mapsto g\cdot x)$ is called \emph{minimal} if, for every non-empty closed subset $F\subseteq X$ such that $\Gamma \cdot F = F$, then $F=X$. A single homeomorphism $\varphi : X \to X$ is minimal if the associated $\mathbb{Z}$ action is minimal.
\end{definition}

Note that $\varphi : X \to X$ is a minimal homeomorphism if and only if for any non-empty closed subset $F \subset X$ with $\varphi(F) = F$, then $F = X$. The next proposition is a standard result in the theory of minimal homeomorphisms.
\begin{proposition} \label{basicMinimalEquivalent}
Suppose that $X$ is compact and $\varphi: X\rightarrow X$ is homeomorphism. Then the following are equivalent:
\begin{enumerate}
\item $\varphi$ is minimal;
\item for each $x\in X$, $\{\ldots, \varphi^{-1}(x), x, \varphi(x), \varphi^2(x), \ldots \}$ is dense in $X$;
\item for each $x\in X$, $\{ x, \varphi(x), \varphi^2(x), \ldots \}$ is dense in $X$;
\item if $U \subseteq X$ is a non-empty open set, then there exists $k\in \N$ such that
\[
U \cup \varphi(U) \cup \ldots \cup \varphi^k(U) =X.
\]
\end{enumerate}
\end{proposition}
In our main results (see Theorems \ref{mainTheoremManifoldCase}, \ref{MainTheoremNonManifoldCase}) there is an additional finite order homeomorphism, $\beta$. We hope that the next result, which is likely known to experts, motivates our assumption that $\beta$ is finite order.
\begin{proposition}
Suppose $X$ is compact metric space, $\varphi: X\rightarrow X$ is a homeomorphism such that for each $n\in \N$, $\varphi^n$ is minimal and $\beta: X \rightarrow X$ is a finite order homeomorphism. If $\beta \circ \varphi =\varphi \circ \beta$, then $\varphi \circ \beta$ is minimal.
\end{proposition}
Before giving the proof of this proposition, note that this result is trivially false without the finite order assumption by simply taking $\beta=\varphi^{-1}$.
\begin{proof}
Let $K$ denote the period of $\beta$ and let $U$ be a non-empty open subset of $X$. Then the fact that $\varphi^K$ is minimal and the previous proposition imply that there exists $L$ such that
\[
X=U \cup \varphi^K(U) \cup \varphi^{2K}(U) \cup \ldots \cup \varphi^{L\cdot K}(U).
\]
Since $\beta^K= \id$ and $\varphi\circ \beta=\varphi\circ \beta$ we have 
\[
(\varphi\circ \beta)^{i\cdot K}=\varphi^{i\cdot K} \circ (\beta^K)^i=\varphi^{i\cdot K}.
\]
for any $i \in \mathbb{N}$.

Using this, we have 
\[
X=U \cup \varphi^K(U) \cup \varphi^{2K}(U) \cup \ldots \cup \varphi^{L\cdot K}(U) \subseteq \bigcup_{j=0}^{L\cdot K} (\varphi\circ \beta)^j(U),
\]
and it follows from the previous proposition that $\varphi\circ \beta$ is minimal.
\end{proof} 

\subsection{Minimal dynamical systems on point-like spaces} \label{constructionZ}
In \cite{DPS:DynZ}, the authors constructed minimal homeomorphisms on  infinite ``point-like'' spaces, that is, infinite compact connected metric spaces that have both the same \v{C}ech cohomology and topological $K$-theory as a point. Moreover,  they have finite covering dimension \cite[Corollary 1.12]{DPS:DynZ}. As these systems will play a main role in the sequel, we review some of their properties.

A generalized cohomology theory is called \emph{continuous} if an inverse limit of spaces induces an inductive limit of groups at the level of the cohomology theory. The interested reader can find more on this notion in \cite[Section 21.3]{Bla:k-theory} (note that in \cite{Bla:k-theory} these notions are formulated in $C^*$-algebraic terms). Two examples of continuous generalized cohomology theories are \v{C}ech cohomology and $K$-theory. $K$-theory is the most relevant generalized cohomology in this paper but our results also apply to \v{C}ech cohomology. 

The existence of minimal diffeomorphisms of odd dimensional spheres of dimension at least three was proved by Fathi and Herman in the uniquely ergodic case \cite{FatHer:Diffeo} and later generalized by Windsor \cite{Wind:not_uniquely_ergo} to minimal diffeomorphisms with a prescribed number of ergodic measures.  In \cite{DPS:DynZ}, the authors showed that such a minimal diffeomorphism can be used to construct minimal dynamical systems on point-like spaces $Z$. Given a minimal diffeomorphism $\varphi : S^d \to S^d$, $d \geq 3$ odd, the associated space $Z$ is constructed by removing a subset $ L_{\infty}$ that is a $\varphi$-invariant immersion of $\mathbb{R}$ in $S^d$,  and completing $S^d \setminus L_{\infty}$ with respect to a metric obtained from the inverse limit structure. A homeomorphism $\zeta : Z \to Z$ is then given by extending the map $\varphi : S^d \setminus L_{\infty} \to S^d \setminus L_{\infty}$ to $Z$.  

There is a factor map $q : Z \to S^d$ which is one-to-one on $S^d \setminus L_{\infty}$, and every $\zeta$-invariant Borel probability measure $\mu$ satisfies $\mu(S^d \setminus L_{\infty}) = 1$. Further details for the factor map can be found in \cite[Corollary 1.16, Lemma 1.14]{DPS:DynZ}, and the reader is directed to the rest of the paper for the general construction of these minimal dynamical systems. We summarize the main aspects in the theorem below.

\begin{theorem} \label{ThmAboutZ}
Let $S^d$  be a sphere with odd dimension $d\geq 3$, and let $\varphi : S^d \to S^d$  be a minimal diffeomorphism. Then there exist an infinite compact metric space $Z$ with covering dimension $d$ or $d-1$ and a minimal homeomorphism \mbox{$\zeta : Z \to Z$} satisfying the following.
\begin{enumerate}
\item $Z$ is compact, connected, and homeomorphic to an inverse limit of compact contractible metric spaces $(Z_n, d_n)_{n \in \mathbb{N}}$.
\item For any continuous generalized cohomology theory we have an isomorphism $H^*(Z) \cong H^*(\{\mathrm{pt}\})$. In particular this holds for \v{C}ech cohomology and \mbox{$K$-theory}. \label{sameCohomAsPoint}
\item There is an almost one-to-one factor map $q : Z \to S^d$ which induces a bijection between $\zeta$-invariant Borel probability measures on $Z$ and $\varphi$-invariant Borel probability measures on $S^d$. \label{FactorMap}
\end{enumerate} 
\end{theorem}

\section{Existence results} \label{Sec:ExistenceResults}

In \cite{GlaWei:MinSkePro}, Glasner and Weiss show how one may obtain skew products systems which are minimal. We will review relevant notation below. In this section, we are interested in skew product systems arising from minimal homeomorphisms on point-like spaces. In the context of the discussion in the introduction, these are existence results. In later sections we will discuss variants where our goal is to put requirements on the induced maps on $K$-theory.

First, let us recall some notation from \cite{GlaWei:MinSkePro}. For a compact metric space $X$ with metric $\d_X$, let $\homeo(X)$ denote the space of homeomorphisms of $X$ equipped with the metric $\d$ given by
\[ \d(g,h) = \sup_{x \in X} \d_X(g(x), h(x)) + \sup_{x \in X} \d_X(g^{-1}(x), h^{-1}(x)).\]

 Let $(Z, \zeta)$ be a minimal dynamical system given by Theorem~\ref{ThmAboutZ}. For a compact metric space $Y$, let $X:=Z\times Y$. Define a subset of ${\homeo}(X)$ by
\[
\mathcal{O}(\zeta\times \id_Y) =\{ G^{-1} \circ (\zeta \times \id_Y )\circ G \mid G \in {\homeo}(X)\}.
\]
Still following \cite{GlaWei:MinSkePro}, we are also interested in subsets of $\mathcal{O}(\zeta\times \id_Y)$. Let ${\rm Homeo}_s(X)$ be the subgroup of ${\rm Homeo}(X)$ consisting of homeomorphisms that fix all subspaces of the form $\{z\} \times Y$ (with $z\in Z$). Notice that if $G\in {\rm Homeo}_s(X)$ then it is determined by a continuous map $Z  \ni z \mapsto g_z\in {\rm Homeo}(Y)$ via $G(z,y)=(z, g_z(y))$. Let
\[
\mathcal{S}(\zeta\times \id_Y) =\{ G^{-1} \circ (\zeta \times \id_Y )\circ G \: | \: G \in {\homeo}_s(X) \}.
\]

\begin{theorem} \label{GlasnerWeissToZ}
Let $(Z, \zeta)$ be a minimal point-like system as in Theorem~\ref{ThmAboutZ}. Suppose that $Y$ is a compact metric space with a path connected subgroup $\Gamma \subset {\rm Homeo}(Y)$ such that $(Y, \Gamma)$ is minimal. Then there exists a residual subset of $\overline{\mathcal{O}(\zeta\times \id_Y)} \subset \homeo(Z \times Y)$ consisting entirely of minimal homeomorphisms. Likewise, there is a residual subset of $\overline{\mathcal{S}(\zeta\times \id_Y)}$ consisting entirely of minimal homeomorphisms.
\end{theorem}
\begin{proof}
The proof is a direct application of Theorem 1 in \cite{GlaWei:MinSkePro}.
\end{proof}

A list of spaces $Y$ that have a path connected subgroup $\Gamma \subset {\rm Homeo}(Y)$ such that $(Y, \Gamma)$ is minimal can be found on page 7 of \cite{DirMal:NonInvMinSkePro}. Although many spaces satisfy this condition, it does not hold for an arbitrary finite CW-complex. However, if $W$ is a finite connected CW-complex, the product of $W$ with the Hilbert cube $Q$ is a connected compact Hilbert cube manifold (see for example \cite[page 498]{West:HilbertCubeMflds}) and hence such a subgroup of ${\rm Homeo}(W \times Q)$ exists. In particular, we may apply Theorem \ref{GlasnerWeissToZ} to $Y =  W\times Q$. Since both the Hilbert cube $Q$ and $Z$ have the same $K$-theory and cohomology as a point, we arrive at the following:

\begin{theorem} \label{Thm:ExistenceOfMinHomeo}
Let $W$ be a finite connected CW-complex, and let $Q$ denote the Hilbert cube. Then $Z\times W\times Q$ admits a minimal homeomorphism and there are  isomorphisms
\[ H^*(Z\times W\times Q) \cong H^*(W), \quad K^*(Z \times W \times Q) \cong K^*(W)\]
of \v{Cech} cohomology and $K$-theory.
\end{theorem}

\begin{proof}
The existence of the minimal homeomorphism follows using Theorem \ref{GlasnerWeissToZ}. The second statement follows from the K\"unneth formula, the fact that $Q$ is contractible, and Theorem \ref{ThmAboutZ} (\ref{sameCohomAsPoint}).
\end{proof}

\begin{remark} \label{uniqueErgodic}
In the case that the system $(Z, \zeta)$ is uniquely ergodic, then  Theorem 2 in \cite{GlaWei:MinSkePro} tells us that  we can obtain a uniquely ergodic skew product system.
\end{remark}

\begin{remark}
If $M$ is a closed connected manifold of finite dimension, then we can apply Theorem \ref{GlasnerWeissToZ} directly to $X=Z\times M$. In this case, $X$ is finite dimensional. In the general case, when we consider finite CW-complexes, we must include the infinite dimensional Hilbert cube and hence the space $X$ is also infinite dimensional.
\end{remark}

\begin{proposition} \label{alphaActZeroKtheory}
Let $W$ be a finite CW-complex, $Q$ the Hilbert cube, and $(Z, \zeta)$ a minimal point-like system. Then any minimal homeomorphism $\alpha \in  \overline{\mathcal{O}(\zeta \times \id_{W\times Q})}$ (or $\overline{\mathcal{S}(\zeta \times \id_{W\times Q})}$) given by Theorem~\ref{GlasnerWeissToZ} induces the identity map on $K$-theory and cohomology.
\end{proposition}

\begin{proof}
We give a detailed proof only for the map on $K$-theory but the proof for the map on cohomology is similar. Since $\alpha$ is in the closure of 
\[ \mathcal{O}(\zeta\times \id_{W\times Q}) =\{ G^{-1} \circ (\zeta \times \id_{W \times Q} )\circ G \: \mid \: G \in {\rm Homeo}(X) \},\] we need only show that elements of $\mathcal{O}(\zeta \times \id_{W \times Q})$ act as the identity map on $K$-theory. This follows since $\zeta^*: K^*(Z) \rightarrow K^*(Z)$ is the identity map, as is shown in the proof of  \cite[Proposition 2.8]{DPS:DynZ}. Thus $(G^{-1} \circ (\zeta \times \id_{W\times Q} )\circ G)^*=G^* \circ \id_{K^*(W\times Q)} \circ (G^*)^{-1}=\id_{K^*(W \times Q)}$.
\end{proof}

A similar but slightly different construction is also possible which will allow us to say more about invariant measures of the minimal dynamical system.   Let $(S^d, \varphi)$, $d \geq 3$ odd,  be a minimal diffeomorphism and $(Z, \zeta)$ the corresponding point-like system given by Theorem~\ref{ThmAboutZ}. For a finite \mbox{CW-complex} $W$ and the Hilbert cube $Q$, we apply \cite[Theorem 1]{GlaWei:MinSkePro} to the product space $S^d\times W \times Q$ to obtain a minimal homeomorphism 
\[
\tilde{\varphi} : S^d\times W \times Q \rightarrow S^d\times W \times Q, \qquad
(s, w, v) \mapsto (\varphi(s), h_s (w, v)),
\]
where $s \in S^d$, $w\in W$, $v \in Q$ and $h : S^d \rightarrow {\rm Homeo}(W\times Q)$. Let $q : Z \to S^d$ be the factor map of Theorem~\ref{ThmAboutZ} (\ref{FactorMap}). Then we define a homeomorphism 
\[ \tilde{\zeta}: Z\times W\times Q \rightarrow Z\times W\times Q, \qquad (z, w, v) \mapsto (\zeta(z), h_{q(z)}(w, v)).
\]

\begin{proposition} \label{FacMapConTwo}
There is a factor map 
\[ \tilde{q} : (Z\times W\times Q, \tilde{\zeta}) \rightarrow (S^d\times W \times Q, \tilde{\varphi}) \]
defined by $\tilde{q}= q \times \id_{W\times Q}$.
\end{proposition}
\begin{proof} Since $q$ is a factor map, it is clear that $\tilde{q}= q \times \id_{W\times Q}$ is surjective. Also,
\begin{eqnarray*} 
\tilde{q} \circ \tilde{\zeta} (z, w, v) &=& \tilde{q}(\zeta(z), h_{q(z)}(w,v)) = (q (\zeta(z)), h_{q(z)}(w,v))\\
&=& (\varphi(q(z)), h_{q(z)}(w,v)) = \tilde{\varphi}(q(z), h_{q(z)}(w,v)),
\end{eqnarray*}
so $\tilde{q}$ intertwines the actions. Thus $\tilde{q}$ is a factor map.
\end{proof}

\begin{proposition} \label{MinMapConTwo}
The homeomorphism \[ \tilde{\zeta}: Z\times W\times Q \rightarrow Z\times W\times Q, \qquad (z, w,v) \mapsto (\zeta(z), h_{q(z)}(w,v))
\] is minimal. 
\end{proposition}
\begin{proof}
Let $L_{\infty} \subset S^d$ denote the $\varphi$-invariant immersion of $\mathbb{R}$ which is removed in the construction of $Z$. Suppose $F$ is a closed non-empty $\tilde{\zeta}$-invariant subset of $Z\times W\times Q$. Then $\tilde{q}(F)$ is a closed non-empty $\tilde{\varphi}$-invariant subset of $S^d\times W \times Q$ and since $\tilde{\varphi}$ is minimal, we have that $\tilde{q}(F)=S^d\times W \times Q$. By \cite[Lemma 1.14]{DPS:DynZ}, $q$ is injective when restricted to the $(S^d \setminus L_{\infty}) \subset Z$, from which it immediately follows that $\tilde{q}$ is injective on $(S^d \setminus L_{\infty}) \times W \times Q$. Hence ($S^d \setminus L_{\infty}) \times W \times Q \subseteq F$. However $(S^d \setminus L_{\infty}) \times W \times Q$ is dense in $Z\times W\times Q$. Thus $F$ is both closed and dense, so we conclude $F=Z\times W\times Q$. It follows that $\tilde{\zeta}$ is minimal.
\end{proof}

\begin{proposition}
The factor map $ \tilde{q} : (Z\times W\times Q, \tilde{\zeta}) \rightarrow (S^d\times W \times Q, \tilde{\varphi})$ induces a bijection between the $\tilde{\zeta}$-invariant Borel probability measures on $Z\times W\times Q$ and the $\tilde{\varphi}$-invariant Borel probability measures on $S^d\times W\times Q$. 
\end{proposition}
\begin{proof}
Let $\mu$ be a $\tilde{\zeta}$-invariant measure. Then ${\tilde{q}}^*(\mu) =\mu\circ \tilde{q}^{-1}$ is $\tilde{\varphi}$-invariant. The set $L_{\infty} \times Q\times W$ is a Borel subset of $S^d \times W \times Q$, and since $L^{\infty}$ is \mbox{$\varphi$-invariant}, $L_{\infty} \times Q\times W$ is invariant under $\tilde{\varphi}$. As in the proof of \cite[Theorem 1.18]{DPS:DynZ}, it follows that $\tilde{q}^*(\mu)(L_{\infty} \times Q\times W)=0$, and hence that  $\mu((S^d \setminus L_{\infty}) \times W \times Q) = 1$. Since the factor map $q: Z \to S^d$ is one-to-one on $S^d \setminus L_{\infty}$, we have that $\tilde{q}$ is bijective on $(S^d \setminus L_{\infty}) \times W \times Q$, and the result follows.
\end{proof}

\begin{proposition}
The minimal homeomorphism 
\[ \tilde{\zeta}: Z\times W\times Q \rightarrow Z\times W\times Q, \qquad (z, w,v) \mapsto (\zeta(z), h_{q(z)}(w,v))
\] 
induces the identity map on $K$-theory and cohomology.
\end{proposition}
\begin{proof}
Using \cite[Proposition 10.5.1]{Bla:k-theory} and the fact that $\zeta^*$ is the identity on $K^*(Z)$, the statement will follow by showing that $\tilde{\zeta}$ is homotopic to a conjugate of the homeomorphism $\zeta\times \id_W \times \id_Q$. The space $S^d\times W \times Q$ is a compact Hilbert cube manifold and hence its homeomorphism group is locally contractible by the main result of \cite{Cha:HomHibCubMfd}. Using this fact together with the skew product construction, there exists a homotopy
\[ H : S^d \times [0,1] \rightarrow {\rm Homeo}(W\times Q), \]
where 
\[ H(z, 0)=h_z \hbox{ and } H(z, 1) = g^{-1}_{\varphi(z)} \circ g_z, \]
for some $g: S^d \rightarrow {\rm Homeo}(W\times Q)$. Precomposing with the factor map $\tilde{q}$ leads to
\[ \tilde{H}: Z\times  [0,1] \rightarrow {\rm Homeo}(W\times Q), \]
where 
\[ \tilde{H}(z, 0)=h_{q(z)} \hbox{ and } \tilde{H}(z, 1) = g^{-1}_{\varphi(q(z))} \circ g_{q(z)} .\]
Summarizing, we have obtained a homotopy from $\tilde{\zeta}$ to the homeomorphism
\[
(z, w) \mapsto (\zeta(z), g^{-1}_{\varphi(q(z))} \circ g_{q(z)}(w) )= (G^{-1} \circ (\zeta \times \id_W \times \id_Q) \circ G)(z,w),
\]
where the homeomorphism $G$ is determined by $g$ via $G(z,w)=(z, g_{q(z)}(w))$. 
\end{proof}

\section{The manifold case}

\subsection{Statement of the result in the manifold case}

If $M$ is a smooth closed manifold, then we let ${\rm Diff}^{\infty}(M)$ denote the collection of smooth self-diffeomorphisms on $M$. Since $M$ is compact, ${\rm Diff}^{\infty}(M)$ is a Fr\'echet Lie group, see Chapter 2 of \cite{Hirsch:DiffTop} for more details. In particular, ${\rm Diff}^{\infty}(M)$ is a complete metric space and the Baire category theorem holds. For an explicit definition of the metric $\d_{\rm{Diff}^{\infty}(M)}$ see \cite[Section 2.4]{Hirsch:DiffTop}. 

\begin{theorem} \label{mainTheoremManifoldCase}
Suppose that $M$ is a smooth closed manifold, $R_{\theta}$ is a free smooth action of $S^1$ on $M$ and $\beta: M \rightarrow M$ is a finite order $C^{\infty}$-diffeomorphism satisfying the following:
\begin{enumerate}
\item[(A)] for each $\theta \in S^1$, $\beta \circ R_{\theta} = R_{\theta} \circ \beta$ (this implies that $R_{\theta}$ gives a well-defined $S^1$-action on the quotient space $M/\beta$);
\item [(B)] the $S^1$-action on the quotient space $M/\beta$ induced by $R_{\theta}$ is free. 
\end{enumerate}
Then there exists $\alpha: M \rightarrow M$ a minimal $C^{\infty}$-diffeomorphism
\[
\alpha = \lim_{n\to \infty} H_n \circ (R_{\theta_n}\circ \beta) \circ H^{-1}_n 
\]
where
\begin{enumerate}
\item for each $n\in \N$, $H_n: M \rightarrow M$ is a $C^{\infty}$-diffeomorphism and $\theta_n \in S^1$;
\item the convergence in the limit occurs within ${\rm Diff}^{\infty}(M)$.
\end{enumerate}
Moreover, $\alpha$ is homotopic to $\beta$ and we can take $\alpha$ to be uniquely ergodic.
\end{theorem}
Before giving the proof, we discuss Conditions (A) and (B), a class of examples satisfying the conditions in the theorem, and a number of more specific examples. 

It is worth noting that Conditions (A) and (B) together imply that $R_{\theta}$ is a free action on $M$. However, we have included the assumption that $R_{\theta}$ is a free action on $M$ in the statement of the theorem to emphasize the connection with \cite[Th\'eor\`eme 1]{FatHer:Diffeo} (see the introduction for more on this connection). Assuming Condition (A), one way for Condition (B) to hold is if the action of $S^1 \times \Z/{\rm ord}(\beta) \Z$ induced by $R_{\theta}$ and $\beta$ is free. However, Condition (B) can be satisfied when this action is not free. In particular, the $\Z/{\rm ord}(\beta) \Z$-action generated by $\beta$ need not be free. The reader might find it useful to consider the special case when $M=S^1 \times N$ with the action of $S^1$ given by rotation on the $S^1$-factor and acts trivially on the $N$-factor, which is a special case of the next example. For an example that does {\bf not} satisfy Condition (B), one can take $M=S^1$ with the $S^1$-action given by rotation and $\beta$ given by rotation by $\pi$. 

\begin{example} \label{manifoldGeneralProductExample} 
Let $N_1$ be a smooth closed manifold admiting a free $S^1$-action $R_{N_1, \theta}$, and let $N_2$ be a smooth closed manifold with a finite order \mbox{$C^{\infty}$-diffeomorphism} $\beta_{N_2}$. Then $R_{\theta}:=R_{N_1, \theta} \times \id_{N_2}$ defines an $S^1$-action on $N_1 \times N_2$ which, together with the finite order $C^{\infty}$-diffeomorphism, $\beta:= \id_{N_1} \times \beta_{N_2}$, satisfies conditions (A) and (B) of Theorem~\ref{mainTheoremManifoldCase}. Many examples can be constructed from this setup.
\end{example}

\begin{example}
Suppose that $1\le p < q$ are odd integers and consider $M=S^p \times S^q$. The $K$-theory of $M$ is given by 
\[
K^0(M) \cong (H^0(S^p) \otimes H^0(S^q)) \oplus (H^p(S^p) \otimes H^q(S^q)) \cong \Z \oplus \Z,
\]
and
\[
K^1(M) \cong (H^0(S^p) \otimes H^q(S^q)) \oplus (H^p(S^p) \otimes H^0(S^q)) \cong \Z \oplus \Z.
\]
The $\Z$-grading on cohomology implies that induced map on $K^0(M) \oplus K^1(M)$ of a homeomorphism on $M$ is one of the following: 
\begin{enumerate}
\item the identity, 
\item $\id \oplus -\id$ in degree zero and $\id \oplus -\id$ in degree one,
\item $\id \oplus -\id$ in degree zero and $-\id \oplus \id$ in degree one,
\item $\id \oplus \id$ in degree zero and $-\id \oplus -\id$ in degree one.
\end{enumerate}
The last of these possible maps on $K$-theory cannot occur for a minimal homeomorphism, because a homeomorphism inducing this map has a fixed point by the Lefschetz fixed point theorem.

We show that the other three possible induced maps can occur in the minimal case. Since $M$ admits a free $S^1$-action, the result of Fathi--Herman discussed in the introduction implies that there exists a minimal diffeomorphism that acts an the identity on $K$-theory. Next, since $S^p$ admits a free $S^1$-action and there is an order two orientation-reversing diffeomeomorphism $\beta: S^q\rightarrow S^q$, the previous theorem implies that there is a minimal homeomorphism that acts on $K$-theory as $\id \oplus -\id$ in degree zero and $-\id \oplus \id$ in degree one. By reversing the roles of $S^p$ and $S^q$ we can also get a minimal diffeomorphism that acts on $K$-theory as $\id \oplus -\id$ in degree zero and $\id \oplus -\id$ in degree one.
\end{example}

\begin{example}
Let $q$ be an odd positive integer and consider $S^q\times \R^n/\Z^n$. Take $B$ an $n$ by $n$ matrix with integer entries that satisfies 
\[ {\rm det}(B)=\pm 1 \hbox{ and }B^L=I \hbox{ for some }L\ge 1. \] 
Note that if a matrix satisfies these conditions, then its characteristic polynomial will be a factor of $z^N-1$ for some $N\in \N$. Many examples can be constructed using this fact. For example, 
\[
B=\left[ \begin{array}{ccc}0 & 0 & -1 \\ 1 & 0 & -1 \\ 0 & 1 & -1 \end{array}\right].
\]
Returning to the general case, given $B$ we define a map on $\R^n/\Z^n$ by $[v] \mapsto [Bv]$ where $v\in \R^n$ and $[v]$ denotes the associated element in $\R^n/\Z^n$. The condition ${\rm det}(B)=\pm 1$ implies that this map is a diffeomorphism and the condition $B^L=I$ implies that it is finite order. Since $S^1$ acts freely on $S^q$ we have the setup of Example \ref{manifoldGeneralProductExample} where $N_1=S^q$, $N_2=\R^n/\Z^n$ and $\beta_{N_2}$ is the finite order diffeomorphism obtained from $B$.

If $q=1$, so that we are considering the $(n+1)$-torus, then there is a minimal diffeomorphism that has action on $H^1(S^1\times \R^n/\Z^n) \cong \Z^{n+1}$ given by $\id \oplus B$.

If $q\ge 3$, then there is a minimal diffeomorphism on $S^q\times \R^n/\Z^n$ whose induced map on $H^1(S^q\times \R^n/\Z^n)\cong H^0(S^q)\otimes H^1(\R^n/\Z^n) \cong \Z^n$ is given by $B$.

Note that if $B\neq I$ then minimal diffeomorphisms obtained via this construction cannot be homotopic to the identity since the map they induced on $K$-theory is not the identity.
\end{example}

\subsection{Proof in the manifold case}

The goal of this section is a proof of Theorem \ref{mainTheoremManifoldCase}. We begin with a number of lemmas. The first is \cite[4.11]{FatHer:Diffeo} and second is based on \cite[4.12]{FatHer:Diffeo}. Recall that the definition of locally free was given in Definition \ref{defLocFree}.

\begin{lemma} \label{FathiHermanLemma1}
Suppose that $M$ is a smooth closed manifold admitting a smooth locally free action of $S^1$. Then, for any non-empty open set $U \subseteq M$, there exists a \mbox{$C^{\infty}$-diffeomorphism} $H$ such that 
\begin{enumerate}
\item $H^{-1}(U)$ meets every orbit of the $S^1$-action;
\item $H$ is homotopic to the identity on $M$.
\end{enumerate}
\end{lemma}
\begin{lemma} \label{FathiHermanLemma2}
Suppose that $M$ is a smooth closed manifold, $R_{\theta}$ is a smooth free action of $S^1$ on $M$ and $\beta: M \rightarrow M$ is a finite order $C^{\infty}$-diffeomorphism satisfying the following:
\begin{enumerate}
\item[(A)] for each $\theta \in S^1$, $\beta \circ R_{\theta} = R_{\theta} \circ \beta$ (this implies that $R_{\theta}$ gives a well-defined $S^1$-action on the quotient space $M/\beta$);
\item [(B)] the $S^1$-action on the quotient space $M/\beta$ induced by $R_{\theta}$ is free. 
\end{enumerate}
Let $\frac{p}{q}$ be a rational number for which the order of $\beta$ divides $q$ and gcd$(p,q)=1$. Then given a non-empty open set $U \subseteq M$, there exists a $C^{\infty}$-diffeomorphism $H$ such that 
\begin{enumerate}
\item $H \circ R_{\frac{p}{q}} \circ \beta \circ H^{-1}=R_{\frac{p}{q}} \circ \beta$;
\item If $m \in M$, then there exists $\theta\in S^1$ and $l=0, \ldots {\rm ord}(\beta)-1$ such that 
\[
R_{\theta}(\beta^l(m)) \in H^{-1}(U);
\]
\item $H$ is homotopic to the identity.
\end{enumerate}
\end{lemma}
\begin{proof}
Fix a non-empty open set $U\subseteq M$.

Let $G$ be the group generated by $R_{\frac{p}{q}} \circ \beta$. Since $R_{\frac{p}{q}} \circ \beta=\beta \circ R_{\frac{p}{q}}$ and the order of $\beta$ divides $q$, $G$ is finite. Moreover, the condition gcd$(p,q)=1$ together with (B) implies that $G$ is the cyclic group of order $q$. 

We let $G$ act on $M$ (via $(R_{\frac{p}{q}} \circ \beta)^k$) and will prove this action is free. Suppose that for $k=0, \ldots, q-1$, $(R_{\frac{p}{q}} \circ \beta)^k(m)=m$. Then
\[
m=(R_{\frac{p}{q}} \circ \beta)^k(m)=R_{\frac{kp}{q}}( \beta^k(m)),
\]
and assumption (B) implies that $\frac{kp}{q}$ is an integer and $k$ is zero or divides the order of $\beta$. Since $k=0, \ldots, q-1$ and gcd$(p,q)=1$, the first of these conditions implies that $k=0$. Hence the action of $G$ on $M$ is free and  $M/G$ is a smooth closed manifold. Furthermore, the quotient map $\pi : M \rightarrow M/G$ is a covering map.

Let $S^1$ act on $M/G$ via $\theta \cdot [m] := [ R_{\theta}(m)]$. We note that since for each $\theta \in S^1$, $\beta \circ R_{\theta} = R_{\theta} \circ \beta$, this action is well defined. Furthermore, since $S^1$ acts freely on $M$ and $G$ is finite, the action of $S^1$ on $M/G$ is locally free.

Applying Lemma \ref{FathiHermanLemma1} to $M/G$ and the open set $\pi(U)$ gives $\bar{H}: M/G \rightarrow M/G$ a $C^{\infty}$-diffeomorphism such that
\begin{enumerate}[resume]
\item $\bar{H}^{-1}(\pi(U))$ meets every $S^1$-orbit, and
\item $\bar{H}$ is homotopic to the identity.
\end{enumerate}
Using (5), we are able to prove that $\bar{H}$ has a lift to a $C^{\infty}$-diffeomorphism $H: M \rightarrow M$. The details are as follows. Consider the map $\bar{H} \circ \pi : M \rightarrow M/G$. Since $\bar{H}$ is homotopic to the identity, the map induced by $\bar{H} \circ \pi$ at the level of fundamental groups is given by $\pi_*$, and hence $\bar{H}$ has a unique continuous lift, $H: M \rightarrow M$. It follows that $H$ is smooth since being smooth is a local property and both $\pi$ and $\bar{H}$ are smooth. Finally, the same argument can be applied to $\bar{H}^{-1} \circ \pi$ to obtain a unique lift of $\bar{H}^{-1}: M/G \rightarrow M/G$. The uniqueness of lifts then implies $H$ is a diffeomorphism (its inverse is the unique lift of $\bar{H}^{-1}$). 

The proof will now be completed by showing that $H$ satisfies the (1), (2) and (3) in the statement of the Theorem.

For (1), by the definition of the lift, we have that $H$ commutes with the action of $G$. In particular,
\[
H \circ (R_{\frac{p}{q}} \circ \beta) \circ H^{-1}= (R_{\frac{p}{q}} \circ \beta) \circ H \circ H^{-1}=(R_{\frac{p}{q}} \circ \beta),
\]
as required.

For (2), let $m \in M$. Since $\bar{H}^{-1}(\pi(U))$ meets every $S^1$-orbit in $M/G$, there exists $\theta\in S^1$ such that $[R_{\theta}(m)] \in \pi(U)$. Using the definition of covering map in the case of $\pi : M \rightarrow M/G$, there exists $g \in G$ such that $g(R_{\theta}(m))\in U$. By the definition of $G$ there exists $k=0, \ldots, q-1$ such that $g=(R_{\frac{p}{q}} \circ \beta)^k$. Hence
\[
R_{\theta+ \frac{kp}{q}}(\beta^k(m))=(R_{\frac{p}{q}} \circ \beta)^k(R_{\theta}(m))\in U,
\]
as required.
For (3), we note that since $\bar{H}$ is homotopic to the identity, then so is its lift $H$.
\end{proof}
We will now prove Theorem \ref{mainTheoremManifoldCase}, which we restate for the reader's convenience:
\begin{theorem*}[Theorem~\ref{mainTheoremManifoldCase}]  
Suppose that $M$ is a smooth closed manifold, $R_{\theta}$ is a smooth free action of $S^1$ on $M$ and $\beta: M \rightarrow M$ is a finite order $C^{\infty}$-diffeomorphism satisfying the following
\begin{enumerate}
\item[(A)] for each $\theta \in S^1$, $\beta \circ R_{\theta} = R_{\theta} \circ \beta$ (this implies that $R_{\theta}$ gives a well-defined $S^1$-action on the quotient space $M/\beta$);
\item [(B)] the $S^1$-action on the quotient space $M/\beta$ induced by $R_{\theta}$ is free. 
\end{enumerate}
Then there exists $\alpha: M \rightarrow M$ a minimal $C^{\infty}$-diffeomorphism
\[
\alpha = \lim_{n\to \infty} H_n \circ (R_{\theta_n}\circ \beta) \circ H^{-1}_n,
\]
where
\begin{enumerate}
\item for each $n\in \N$, $H_n: M \rightarrow M$ is a $C^{\infty}$-diffeomorphism and $\theta_n \in S^1$;
\item the convergence in the limit occurs within ${\rm Diff}^{\infty}(M)$.
\end{enumerate}
Moreover, $\alpha$ is homotopic to $\beta$ and we can take $\alpha$ to be uniquely ergodic.
\end{theorem*}

\begin{proof}
The proof follows the structure of the proof of Theorem 1 from \cite{FatHer:Diffeo}, which begins on page 15 of that paper. The general idea is due to to Anosov and Katok \cite{MR0370662}.

Consider 
\[
\mathcal{O}^{\infty}(S^1, \beta):=\{ g \circ (R_t \circ \beta) \circ g^{-1} \mid t\in S^1, g \in {\rm Diff}^{\infty}(M) \}
\]
as a subset of Diff$^{\infty}(M)$. Since Diff$^{\infty}(M)$ is a complete metric space, the Baire category theorem holds for $\overline{\mathcal{O}^{\infty}(S^1, \beta)}$.

Given a non-empty open set $U$, define
\[
E_U:= \{ f \in \overline{\mathcal{O}^{\infty}(S^1, \beta)} \mid U \cup f(U) \cup \ldots \cup f^L(U)=M \text{ for some } L \in \mathbb{N} \}.
\]
For each non-empty open set $U$, the set $E_U$ is open in $\overline{\mathcal{O}^{\infty}(S^1, \beta)}$. 

We will show that for each non-empty open set $U$, the set $E_U$  is dense. By construction $\mathcal{O}^{\infty}(S^1, \beta)$ is dense in $\overline{\mathcal{O}^{\infty}(S^1, \beta)}$ so we need only show that, for each non-empty open set $U$, $g \circ (R_t \circ \beta) \circ g^{-1} \in \overline{E_U}$ for each $t\in S^1$ and $g \in {\rm Diff}^{\infty}(M)$. However,
\[
g \circ (R_t \circ \beta) \circ g^{-1} \in \overline{E_U} \hbox{ if and only if }R_t\circ \beta \in \overline{E_{g^{-1}(U)}}.
\]
Since $g^{-1}(U)$ is a non-empty open set, this reduces the proof to showing that for each non-empty open set $U$ and $t\in S^1$, $R_t \circ \beta \in \overline{E_U}$. Letting $D$ denote a dense subset of $S^1$, we can reduce further to proving that for each non-empty open set $U$ and $t\in D$, $R_t \circ \beta \in \overline{E_U}$.

We now specify the dense subset of $S^1$ that we will consider. Let
\[
D:= \left\{ \frac{p}{q} \in S^1 \mid \hbox{ the order of }\beta \hbox{ divides }q \hbox{ and }{\rm gcd}(p,q)=1 \right\}.
\]
The fact that $D$ is dense follows from a standard argument using the fact that $D$ contains elements that are arbitrarily small.
 
We now fix a non-empty open set $U$, $\frac{p}{q} \in D$ and a sequence of irrational numbers $\theta_n$ that converge to $\frac{p}{q}$. Applying Lemma \ref{FathiHermanLemma2} we obtain a $C^{\infty}$-diffeomorphism $H$ such that 
\begin{enumerate}[resume]
\item $H \circ R_{\frac{p}{q}} \circ \beta \circ H^{-1}=R_{\frac{p}{q}} \circ \beta$;
\item if $m \in M$, then there exists $\theta\in S^1$ and $l=0, \ldots, {\rm ord}(\beta)-1$ such that 
\[
R_{\theta}(\beta^l(m)) \in H^{-1}(U);
\]
\item $H$ is homotopic to the identity.
\end{enumerate}
The fact that $\theta_n \rightarrow \frac{p}{q}$ as $n \rightarrow \infty$ implies that
\[
H \circ R_{\theta_n} \circ \beta \circ H^{-1} \rightarrow H \circ R_{\frac{p}{q}} \circ \beta \circ H^{-1}=R_{\frac{p}{q}} \circ \beta.
\]
Therefore we need only show that $H \circ (R_{\theta} \circ \beta) \circ H^{-1} \in E_U$ for each irrational $\theta\in S^1$. The definition of $E_U$ and the compactness of $M$ reduces the proof to showing that 
\[
\cup_{i=0}^{\infty} (H \circ R_{\theta} \circ \beta \circ H^{-1})^i(U) = M,
\]
which is equivalent to showing
\[
\cup_{i=0}^{\infty} (R_{\theta} \circ \beta)^i (H^{-1}(U)) = M.
\]
Let $m \in M$ and set
\[
V:=\{ R_\theta (\beta^l(m)) \mid \theta \in S^1, l\in \N \}.
\]
Notice that $V$ is $(R_\theta \circ \beta)$-invariant and is diffeomorphic to the disjoint union of $k$-circles where $k \in \N$ satisfies
\begin{enumerate}[resume]
\item $\beta^{k+1}(m)=m$ and
\item $\beta^l(m)\neq m$ for $1\le l \le k$.
\end{enumerate}
We identify $V \cong S^1\times \{ 0, \ldots, k \}$. Then, under this identification, we have that for any $\theta$, $R_{\theta}\circ \beta$ acts via 
\[
(z, i) \in S^1 \times \{ 0, \ldots, k\} \mapsto ({\rm Rot}_{\theta}(z), i+1 \hbox{ mod }k+1)
\]
where ${\rm Rot}_{\theta}$ denotes rotation by $\theta$. In particular, for $\theta$ irrational, $(R_{\theta}\circ \beta)|_V$ is minimal.

By (5), there exists $\theta\in S^1$ and $l=0, \ldots, {\rm ord}(\beta)-1$ such that 
\[
R_{\theta}(\beta^l(m)) \in H^{-1}(U).
\]
Thus $G=H^{-1}(U) \cap V$ is a non-empty set and moreover is open in the subspace topology of $V$ because $H^{-1}(U)$ is open in $M$. Since $(R_{\theta}\circ \beta)|_V$ is minimal there exists $L\in \N$ such that
\[
\cup_{i=0}^L (R_{\theta}\circ \beta)^i(G) = V.
\]
Since $m\in V$, it follows that
\[
m \in \cup_{i=0}^L (R_{\theta}\circ \beta)^i(G) \subseteq \cup_{i=0}^L (R_{\theta}\circ \beta)^i(H^{-1}(U)) \subseteq \cup_{i=0}^{\infty} (R_{\theta} \circ \beta)^i(H^{-1}(U)).
\]
The choice of $m$ was arbitrary, so for each $\theta$ irrational, $H\circ (R_{\theta}\circ \beta)\circ H^{-1} \in E_U$ and hence, as observed above, we have that $E_U$ is dense in $ \overline{\mathcal{O}^{\infty}(S^1, \beta)}$.

To complete the proof, let $\{ U_i \}_{i\in \N}$ be a countable basis for the topology on $M$. Then by the Baire category theorem
\[
\cap_{i=0}^{\infty} E_{U_i}
\]
is dense in $\overline{\mathcal{O}^{\infty}(S^1, \beta)}$ and it follows from Proposition \ref{basicMinimalEquivalent} that each element in $\cap_{i=0}^{\infty} E_{U_i}$ is minimal.

The proof that we can take $\alpha$ to be homotopic to the $\beta$ is similar to the proof of Theorem \ref{homotopyLemmaNonMfldCase} so we omit the details.

Finally, the argument that we can take $\alpha$ to be uniquely ergodic is the same as in the proof given in \cite[Section 3]{GlaWei:MinSkePro}.
\end{proof}

\section{Generalization to homogeneous metric spaces}

To obtain progress on Question B from the introduction, a variant of the main result in the previous section is required. This variant will also be a key result for the $C^*$-algebraic applications in \cite{DPS:minDSKth}. In particular, we must move outside of the smooth category. Our result is related to a theorem of Glasner--Weiss \cite[Theorem 1]{GlaWei:MinSkePro}. More precisely, we extend the special case of their result (Theorem \ref{GlasnerWeissToZ} above) so as to allow for more general induced maps on $K$-theory. It is worth noting that the proof techniques we use are still closely related to work in \cite{FatHer:Diffeo} by Fathi--Herman and therefore build on the work of Anosov and Katok \cite{MR0370662}.

\subsection{The relevant complete metric space}

Let $M$ be a smooth closed manifold and let $Y$ be a compact metric space. We fix a Riemannian metric on $M$ and take the $l_1$-product metric on $M\times Y$. Define $\mathcal{A}$ to be the collection of homeomorphisms $\psi \in {\rm Homeo}(M\times Y)$ of the form
\[
 \psi(m,y)= (\alpha(m), g_m(y)),
 \]
 where 
 \begin{enumerate}
 \item $\alpha \in {\rm Diff}^{\infty}(M)$, and 
 \item $m \mapsto g_m$ is a continuous map from $M$ to ${\rm Homeo}(Y)$.
 \end{enumerate}
Define a metric on $\mathcal{A}$ via
\[
\d_{\mathcal{A}}(\psi, \tilde{\psi})= \d_{{\rm Diff}^{\infty}(M)}(\alpha, \tilde{\alpha})+ \d_{{\rm Homeo}(M\times Y)}(\psi, \tilde{\psi}).
\]
The proof of the next theorem uses the following straightforward observations:
\begin{enumerate}
\item The metric $\d(g, \tilde{g}) :=\sup_{m\in M} \d_{{\rm Homeo}(Y)}(g_m, \tilde{g}_m)$ is complete, and
\item $\d(g, \tilde{g}) \le \d_{\mathcal{A}}(\psi, \tilde{\psi})$.
\end{enumerate}
\begin{theorem}
$(\mathcal{A}, \d_{\mathcal{A}})$ is a complete metric space.
\end{theorem}
\begin{proof}
Checking that $\d_{\mathcal{A}}$ is a metric is standard and the details are therefore omitted.  

To show that it $\mathcal{A}$ is complete with respect to $\d_A$, let $\{ \psi_k \}_{k=1}^{\infty}$ be a Cauchy sequence in $\mathcal{A}$. By the definition of $\mathcal{A}$, $\psi_k(m,y)=(\alpha_k(m), g^{(k)}_m(y))$ for some $\alpha_k \in {\rm Diff}^{\infty}(M)$ and continuous maps $g^{(k)}_m : M \to {\rm Homeo}(Y)$. The definition of $\d_{\mathcal{A}}$ implies that $\{\alpha_k\}_{k=1}^{\infty}$ is a Cauchy sequence in ${\rm Diff}^{\infty}(M)$. Since ${\rm Diff}^{\infty}(M)$ is complete, there exists $\alpha \in {\rm Diff}^{\infty}(M)$ such that $\alpha = \lim_{k \to \infty} \alpha_k$. 

Similarly,  $({\rm Homeo}(M\times Y), \d_{\homeo(M \times Y)})$ is complete so the definition of $\d_{\mathcal{A}}$ implies there exists $\psi \in {\rm Homeo}(M\times Y)$ such that $\psi = \lim_{k \to \infty} \psi_k$ with respect to $\d_{\homeo(M \times Y)}$.

Now, the fact that $\{ \psi_k \}_{k=1}^{\infty}$ is Cauchy with respect to $\d_{{\rm Homeo}(M\times Y)}$ and, as observed above, $\d(g^{(k)}, {g}^{(l)}) :=\sup_{m\in M} \d_{{\rm Homeo}(Y)}(g^{(k)}_m, g^{(l)}_m) \leq \d_{\mathcal{A}}(\psi_k, \psi_l)$ implies that $\{g^{(k)}\}_{k=1}^{\infty}$ is Cauchy. By the completeness of $\d$, there exists $g: M \rightarrow {\rm Homeo}(Y)$ such that $g = \lim_{k \to \infty} g^{(k)}$.

Finally, we must show that $\psi \in \mathcal{A}$  and $\psi_k \rightarrow \psi$  with respect to $\d_{\mathcal{A}}$. It is straightforward to show that $\psi(m,y)=(\alpha(m), g_m(y))$, so $\psi \in \mathcal{A}$. That $\psi_k \rightarrow \psi$ now follows from the definition $\d_\mathcal{A}$.
\end{proof}

\subsection{Generalization to the (possibly) non-manifold case}

\begin{lemma} \label{FathiHermanLemmaToNopenPart1}
Suppose that $M$ is a smooth closed manifold admitting a locally smooth free action of $S^1$. Then, given a non-empty open set $U \subseteq M$ and $N\in \N$, there exist a collection of open sets $\{U_i\}_{i=0}^N$ and a $C^{\infty}$-diffeomorphism $H$ such that 
\begin{enumerate}
\item for each $i=0, \ldots, N$, $\overline{U_i} \subseteq U$;
\item for each $i\neq j$, $\overline{U_i} \cap \overline{U_j} = \emptyset$;
\item for each $i=0, \ldots, N$, $H^{-1}(U_i)$ meets each orbit of the $S^1$-action;
\item $H$ is homotopic to the identity on $M$.
\end{enumerate}
\end{lemma}
\begin{proof}
Before beginning the proof, notice that the $N=0$ case is exactly \cite[Corollaire 4.11]{FatHer:Diffeo}, which was stated above as Lemma~\ref{FathiHermanLemma1}. As such, we follow the proof of \cite[Corollaire 4.11]{FatHer:Diffeo} closely (see also \cite[Proposition 4.8]{FatHer:Diffeo}). 

Let $d$ denote the dimension of $M$. The locally free $S^1$-action gives $M$ the structure of a foliation with codimension $d-1$. As in the proof of \cite[Corollaire 4.11]{FatHer:Diffeo}, there exist $k\ge 1$, closed $d-1$ dimensional disks $\{ D_j \}_{j=1}^k$, $\epsilon>0$, and $C^{\infty}$-diffeomorphism $H:M \rightarrow M$ such that
\begin{enumerate}
\item[(I)] the disks are pairwise disjoint and transverse to the the foliation;
\item[(II)] the union of the interiors of the disks meet each $S^1$-orbit;
\item[(III)] $H( \cup_{j=1}^k D_j\times [0, \epsilon]) \subseteq U$ where for each $j$, $D_j \times [0, \epsilon]$ is identified with a closed disk of dimension $d$ inside $M$;
\item[(IV)] $H$ is homotopic to the identity.
\end{enumerate} 
Next for each $i=0, 1, \ldots N$, let
\[
U_i :=H \left(\bigcup_{j=1}^k {\rm int}(D_j) \times \left( \frac{2i \epsilon}{4N}, \frac{(2i+1) \epsilon}{4N} \right) \right).
\]
Using (I)--(IV),  one sees that $\{ U_i \}_{i=0}^N$ and $H$ have the required properties, namely (1)--(4) in the statement of the lemma.
\end{proof}

Next, we generalize \cite[Corollaire 4.12]{FatHer:Diffeo}.
\begin{lemma} \label{FathiHermanLemmaToNopen}
Suppose that $M$ is a smooth closed manifold admitting a smooth locally free action of $S^1$. Then, given a rational number $\frac{p}{q}$, a non-empty open set $U \subseteq M$ and $N\in \N$, there exists a collection of open sets $\{U_i\}_{i=0}^N$ and a $C^{\infty}$-diffeomorphism $H$ such that 
\begin{enumerate}
\item for each $i=0, \ldots, N$, $\overline{U_i} \subseteq U$;
\item for each $i\neq j$, $\overline{U_i} \cap \overline{U_j} = \emptyset$;
\item for each $i=0, \ldots, N$, $H^{-1}(U_i)$ meets each of orbit of the $S^1$-action;
\item $H \circ R_{\frac{p}{q}} \circ H^{-1}=R_{\frac{p}{q}}$.
\end{enumerate}
\end{lemma}
\begin{proof}
The proof is the same as the proof of \cite[Corollaire 4.12]{FatHer:Diffeo} upon replacing the use of  \cite[Corollaire 4.11]{FatHer:Diffeo} with Lemma \ref{FathiHermanLemmaToNopenPart1}. We omit the details.
\end{proof}

Suppose in addition to the above notation, $\beta$ is a finite order homeomorphism of $Y$. Let $\mathcal{S}(S^1, \beta)$ be the collection of homeomorphisms $M\times Y \to M \times Y$ of the form
\[ 
\psi=h \circ (R_t \times \beta) \circ h^{-1},
\]
for some $h\in \mathcal{A}$ and $t\in S^1$.

One can check that $\mathcal{S}(S^1, \beta) \subseteq \mathcal{A}$.  Since $\mathcal{A}$ is a complete metric space, the Baire category theorem holds in the closure $\overline{\mathcal{S}(S^1, \beta)} \subset \mathcal{A}$.

\begin{theorem} \label{MainTheoremNonManifoldCase}
Suppose that $M$ is a smooth closed manifold, $R_{\theta}$ is a smooth free action of $S^1$ on $M$, $Y$ is a compact metric space, $\Gamma$ is a path connected subgroup of $\homeo(Y)$ such that $(Y, \Gamma)$ is minimal, and $\beta: Y \rightarrow Y$ is a finite order homeomorphism.
Then there is a dense $G_{\delta}$-set of uniquely ergodic minimal homeomorpisms in $\overline{\mathcal{S}(S^1, \beta)}$. 
\end{theorem}
\begin{remark} \label{RemMainTheoremNonManifoldCase}
It is worth noting that $Y$ in the previous theorem can be taken to be any closed, connected manifold or any compact, connected Hilbert manifold, see page 323 of \cite{GlaWei:MinSkePro}. The relevant $\Gamma$ is the path component of the identity in $\homeo(Y)$.
\end{remark}
\begin{proof}
Given $W$ a non-empty open set in $M\times Y$, we define $E_{W}$ to be the collection 
\[ \{f\in \overline{\mathcal{S}(S^1, \beta)} \mid W \cup f(W) \cup \ldots \cup f^L(W)=M \times Y, \hbox{for some } L \in \mathbb{N}\}.  \]
Note that for each $W$, $E_{W}$ is open. We will show it is dense in $\overline{\mathcal{S}(S^1, \beta)}$. Define
\[
D= \left\{ \frac{p}{q} \in S^1 \mid \hbox{the order of }\beta \hbox{ divides }q \hbox{ and }{\rm gcd}(p,q)=1 \right\}.
\]

As in the proof of Theorem \ref{mainTheoremManifoldCase}, we need only prove that for each non-empty open set $W$ and $\frac{p}{q}\in D$, we have that $R_{\frac{p}{q}} \times \beta \in \overline{E_W}$. Furthermore, since we are working with the product topology, we need only prove that $R_{\frac{p}{q}} \times \beta \in \overline{E_{U\times V}}$ for each non-empty open set $U$ in $M$ and each non-empty open set $V$ in $Y$.

As such, we fix $W=U\times V$ where $U$ is a non-empty open set in $M$ and $V$ is a non-empty open set in $Y$. By assumption $(Y, \Gamma)$ is minimal and hence there exists $h_1, \ldots, h_N$ in $\Gamma$ such that 
\[
V \cup h_1(V) \cup \ldots \cup h_N(V) = Y.
\]

Using the fact that the action of $S^1$ is free, and by possibly replacing $U$ with a smaller non-empty open subset, we can assume that the sets
\[
\overline{U}, \overline{R_{\frac{p}{q}}(U)}, \ldots, \overline{R_{\frac{p}{q}}^{q-1}(U)}
\]
are pairwise disjoint.

Apply Lemma~\ref{FathiHermanLemmaToNopen} to $U$ to obtain open subsets $\{ U_i \}_{i=0}^N$ and a $C^{\infty}$-diffeomorphism  $H$ satisfying (1)--(4) of Lemma~\ref{FathiHermanLemmaToNopen}. Our goal is to define a continuous map $g: M \rightarrow \Gamma$. We start with elements in $M$ that are in $U_i$ for some $i = 0, \dots, N$. Let 
\[ g_{m} = h_i \hbox{ whenever }m\in U_i,
\]
where $h_0:=\id_Y$.
We then extend $g$ continuously to a map $U \rightarrow  \Gamma$ satisfying $g_m= \id_Y$ for any $m$ in the boundary of $U$. This can be done because each $h_i$ is homotopic to the identity.

Next, we extend $g$ to 
\[
m \in R_{\frac{p}{q}}(U) \dot{\cup} \ldots \dot{\cup}  R_{\frac{p}{q}}^{q-1}(U).
\]
Notice that if $m$ is in this set, then there exist  unique $\tilde{m} \in U$ and unique $k$,  $1\le k \le q-1$, satisfying
\[
R_{\frac{p}{q}}^k(\tilde{m})=m.
\]
Define
\[
g_{m} := \beta^k \circ g_{\tilde{m}} \circ \beta^{-k}.
\]
Finally, we extend $g$ to all of $M$ by defining it to be the identity for any point not in $U \dot{\cup} R_{\frac{p}{q}}(U) \dot{\cup} \ldots \dot{\cup}  R_{\frac{p}{q}}^{q-1}(U)$. One can check that $g: M \rightarrow \Gamma$ is well defined and continuous. 

Let $G \in \homeo(M \times Y)$ be defined by $G(m,y):=(m, g_m(y))$. For $t\in S^1$ consider $G^{-1} \circ (H\times \id_Y) \circ (R_t \times \beta) \circ (H^{-1}\times \id_Y) \circ G$. Observe that $G^{-1} \circ (H \times \id_Y) \circ (R_t \times \beta) \circ (H^{-1} \times \id_Y) \circ G \in \mathcal{S}(S^1, \beta)$. 

Our next goal is to show that 
\[
G^{-1} \circ (H\times \id_Y) \circ (R_{\frac{p}{q}} \times \beta) \circ (H^{-1}\times \id_Y) \circ G=R_{\frac{p}{q}} \times \beta.
\]
Lemma \ref{FathiHermanLemmaToNopen} implies that 
\[ G^{-1} \circ (H\times \id_Y) \circ (R_{\frac{p}{q}} \times \beta) \circ (H^{-1}\times \id_Y) \circ G=G^{-1} \circ (R_{\frac{p}{q}} \times \beta) \circ G. \]
Let $(m,y)\in M \times Y$. There are two cases:

{\bf Case 1: } Suppose that $m \in U \dot{\cup} R_{\frac{p}{q}}(U) \dot{\cup} \ldots \dot{\cup}  R_{\frac{p}{q}}^{q-1}(U)$. Then 
\begin{align*}
\lefteqn{G^{-1} \circ (R_{\frac{p}{q}} \times \beta) \circ G(m,y)}\\
 & = (R_{\frac{p}{q}}(m), g_{R_{\frac{p}{q}}(m)}^{-1}(\beta(g_m(y)))) \\
& =(R_{\frac{p}{q}}(m),((\beta^{k+1} \circ g_{\tilde{m}} \circ \beta^{-k-1})^{-1} \circ \beta \circ \beta^k \circ g_{\tilde{m}}\circ \beta^{-k})(y)) \\
& =(R_{\frac{p}{q}}(m),(\beta^{k+1} \circ g_{\tilde{m}}^{-1} \circ \beta^{-k-1} \circ \beta \circ \beta^k \circ g_{\tilde{m}}\circ \beta^{-k})(y)) \\
& = (R_{\frac{p}{q}}(m), \beta(y)).
\end{align*}

{\bf Case 2: } Suppose that $m \not\in U \dot{\cup} R_{\frac{p}{q}}(U) \dot{\cup} \ldots \dot{\cup}  R_{\frac{p}{q}}^{q-1}(U)$. Then we must also have that $R_{\frac{p}{q}}(m) \not\in U \dot{\cup} R_{\frac{p}{q}}(U) \dot{\cup} \ldots \dot{\cup}  R_{\frac{p}{q}}^{q-1}(U)$. Thus, by the definition of $g: M \rightarrow \Gamma$, we have that $g_{m}=\id_Y$ and $g_{R_{\frac{p}{q}}(m)}=\id_Y$, and we have shown that \[
G^{-1} \circ (H\times \id_Y) \circ (R_{\frac{p}{q}} \times \beta) \circ (H^{-1}\times \id_Y) \circ G=R_{\frac{p}{q}} \times \beta.\]

Take a sequence of irrational numbers $\{\theta_n\}_{n=0}^\infty$ that converge to $\frac{p}{q}$. Then 
\[
G^{-1} \circ (H\times \id_Y) \circ (R_{\theta_n} \times \beta) \circ (H^{-1}\times \id_Y) \circ G
\]
converges to 
\[
G^{-1} \circ (H\times \id_Y) \circ (R_{\frac{p}{q}} \times \beta) \circ (H^{-1}\times \id_Y) \circ G=R_{\frac{p}{q}} \times \beta.
\]
Hence to prove that $R_{\frac{p}{q}} \times \beta\in \overline{E_{U\times V}}$, we need only show that for each $\theta$ irrational, $G^{-1} \circ (H\times \id_Y) \circ (R_{\theta_n} \times \beta) \circ (H^{-1}\times \id_Y) \circ G \in E_{U\times V}$. By compactness of $M\times Y$, this further reduces to showing that
\[
\cup_{i=0}^{\infty}(R_{\theta} \times \beta)^i((H^{-1} \times \id_Y) \circ G)(U\times V))=M \times Y.
\]
Let $(m,y)\in M \times Y$. There exists $0\le j_0 \le N$ and $\tilde{y}\in V$ such that $h_{j_0}(\tilde{y})=y$. Since $H^{-1}(U_{j_0})$ meets each $S^1$-orbit, there exists $\tilde{m} \in U_{j_0}$ and $t_0\in S^1$ such that $R_{t_0}(m)=H^{-1}(\tilde{m})$. Then 
\[
((H^{-1}\times \id_Y)\circ G)(\tilde{m}, \tilde{y}) = (H^{-1}\times \id_Y)(\tilde{m}, g_{\tilde{m}}(\tilde{y}))=(H^{-1}(\tilde{m}), h_{j_0}(\tilde{y}))=(R_{t_0}(m), y)
\]
Hence, we have shown that
\[
((H^{-1}\times \id_Y) \circ G)(U \times V) \cap \{ (R_{t}(m), y) \mid t \in S^1 \}
\] 
is a non-empty open set in $\{ (R_{t}(m), y) \mid t \in S^1 \}$.

The set $\{ (R_{t}(m), y) \mid t \in S^1 \}$ is $(R_t\times \id_Y)$-invariant, and for fixed $t$, the restriction of $R_t\times \id_Y$ to this set is given by rotation of $t$. If $\theta$ is irrational, then ${\rm ord}(\beta)\cdot \theta$ is also irrational and hence there exists $L\in \N$ such that
\begin{align*}
&\cup_{i=0}^L (R_{{\rm ord}(\beta)\cdot \theta} \times \id_Y)^i(H^{-1}(U) \times G(V) \cap \{ (R_{t}(m), y) \mid t \in S^1 \}) \\
&\quad =\{ (R_{t}(m), y) \mid t \in S^1 \}.
\end{align*}
In particular, $(m, y) \in \cup_{i=0}^L (R_{{\rm ord}(\beta)\cdot \theta} \times \id_Y)^i(H^{-1}(U) \times G(V))$. So
\begin{align*}
(m, y) & \in \cup_{i=0}^L (R_{\theta} \times \id_Y)^{ {\rm ord}(\beta) \cdot i}(H^{-1}(U) \times G(V)) \\
& \subseteq \cup_{i=0}^{{\rm ord}(\beta)\cdot L} (R_{\theta} \times \beta)^i(H^{-1}(U) \times G(V)) \\
& \subseteq \cup_{i=0}^{\infty} (R_{\theta} \times \beta)^i(H^{-1}(U) \times G(V)) .
\end{align*}
Since $(m,y)$ was arbitrary, it follows that 
\[ 
\cup_{i=0}^{\infty}(R_{\theta} \times \beta)^i(H^{-1}(U) \times G(V))=M \times Y,
\]
whence $G^{-1} \circ (H\times id_Y) \circ (R_{\theta} \times \beta) \circ (H^{-1}\times id_Y) \circ G \in E_{U\times V}$ for each $\theta$ irrational. Hence we have that $R_{\frac{p}{q}}\times \beta \in \overline{E_{U\times V}}$ and $E_{U\times V}$ is dense in $\overline{\mathcal{S}(S^1, \beta)}$. The statement of the theorem now follows from the Baire category theorem and the fact that topology on $M\times Y$ has a countable basis. Also note that the uniquely ergodic part of the theorem follows in the same way as in the proof given in \cite[Section 3]{GlaWei:MinSkePro}.
\end{proof}
\begin{theorem} \label{SmoothMap} Let $\alpha \in \diff^\infty(M)$ and let $m \mapsto h_m$ be a continuous map from $M$ to ${\rm Homeo}(Y)$. Let $\psi\in \overline{\mathcal{S}(S^1, \beta)} \subseteq \mathcal{A}$ be defined by $\psi(m,y)=(\alpha(m), h_m(y))$, $(m,y) \in M \times Y$. Then if $\psi$ is minimal,  $\alpha$ is also minimal.
\end{theorem}
\begin{proof}
Let $U$ be a non-empty open subset of $M$. Then since $\psi$ is minimal, there exists $L\in \N$ such that
\[
\cup_{i=0}^L \psi^i(U\times Y) = M \times Y.
\]
This implies that
\[
\cup_{i=0}^L \alpha^i(U) = M.
\]
Since $U$ was an arbitrary non-empty open set, it follows that $\alpha$ is minimal.
\end{proof}

Recall that a topological space is locally contractible if each point in the space has a local basis of contractible sets.
\begin{theorem} \label{homotopyLemmaNonMfldCase}
Suppose that $M$ is a smooth closed manifold, $R_{\theta}$ is a smooth free action of $S^1$ on $M$, $Y$ is a compact metric space, $\Gamma$ is a path connected subgroup of $\homeo(Y)$ such that $(Y, \Gamma)$ is minimal and $\beta: Y \rightarrow Y$ is a finite order homeomorphism.
 If the homeomorphism group of $M\times Y$ is locally contractible, then there exists $\psi \in \overline{\mathcal{S}(S^1, \beta)}$ such that $\psi$ is homotopic to $\id_M \times \beta$, minimal, and uniquely ergodic.
\end{theorem}
\begin{proof}
Since ${\rm Homeo}(M\times Y)$ is locally contractible, there exists an open neighborhood of $\id_M \times \beta$, $U$, such that all homeomorphisms in $U$ are homotopic to the $\id_M\times \beta$. 

By Theorem \ref{MainTheoremNonManifoldCase}, the set
\[ \{ \psi \in  \overline{\mathcal{S}(S^1, \beta)} \mid \psi \hbox{ is minimal and uniquely ergodic} \} 
\]
is dense in $\overline{\mathcal{S}(S^1, \beta)}$. By definition, $\id_M \times \beta \in \overline{\mathcal{S}(S^1, \beta)}$ so there exists a sequence of uniquely ergodic minimal homeomorphisms converging to $\id_M \times \beta$ with respect to the metric $d_{\mathcal{A}}$. In particular, this sequence converges to $\id_M\times \beta$ with respect to $d_{{\rm Homeo}(M\times Y)}$. Hence there exists a uniquely ergodic minimal homeomorphism in $U$. This completes the proof as all elements in $U$ are homotopic to $\id_M \times \beta$.
\end{proof}

\subsection{Constructions from minimal homeomorphisms on point-like spaces} \label{PointLikeSpacesWithBeta}
Let $d\ge 3$ be an odd integer, $Q$ be the Hilbert cube, and $W$ a connected finite \mbox{CW-complex} with finite order homeomorphism $\beta_W: W\rightarrow W$. Then, as observed in Remark \ref{RemMainTheoremNonManifoldCase}, we can apply Theorem \ref{MainTheoremNonManifoldCase} with $M=S^d$, $Y= W \times Q$, and $\beta= \beta_W \times \id_Q$ to obtain a minimal homeomorphism 
\[
\tilde{\beta} : S^d\times W \times Q \rightarrow S^d\times W \times Q, \qquad
(s, w) \mapsto (\varphi(s), h_s (w)),
\]
where $s \in S^d$, $w\in W\times Q$, $\varphi \in \diff^\infty(S^d)$, and $h : S^d \rightarrow {\rm Homeo}(W\times Q)$ continuous. Moreover, $h$ is of the form
\[
h_s =g^{-1}_{\varphi(s)} \circ \beta \circ g_s,
\] 
where $g: S^d \rightarrow {\rm Homeo}(W\times Q)$ is a continuous map. By Theorem \ref{SmoothMap}, the diffeomorphism $\varphi: S^d \rightarrow S^d$ is minimal and therefore, using Theorem~\ref{ThmAboutZ}, there exists a corresponding minimal point-like system $(Z, \zeta)$. 

Let $q : Z \to S^d$ be the factor map of Theorem~\ref{ThmAboutZ}~(\ref{FactorMap}). Then we define a homeomorphism 
\[ \tilde{\zeta}: Z\times W\times Q \rightarrow Z\times W\times Q, \qquad (z, w) \mapsto (\zeta(z), h_{q(z)}(w)).
\]

The next two propositions are analogous to Propositions~\ref{FacMapConTwo} and \ref{MinMapConTwo}, now applied to $h$ as above. The proofs carry through verbatim, so are omitted. 

\begin{proposition} \label{FacMapConTwoInGeneral}
There is a factor map 
\[ \tilde{q} : (Z\times W\times Q, \tilde{\zeta}) \rightarrow (S^d\times W \times Q, \tilde{\beta}) \]
defined by $\tilde{q}= q \times \id_{W\times Q}$.
\end{proposition}

\begin{proposition} \label{MinMapConTwoInGeneral}
The homeomorphism \[ \tilde{\zeta}: Z\times W\times Q \rightarrow Z\times W\times Q, \qquad (z, w) \mapsto (\zeta(z), h_{q(z)}(w))
\] is minimal. 
\end{proposition}

\begin{proposition} \label{SameMapOnKtheory}
The map on $K$-theory induced by the minimal homeomorphism $\tilde{\zeta}$ satisfies the following
\begin{displaymath} 
\xymatrix{ K^*(Z\times W\times Q)  \ar[r]^{\tilde{\zeta}^*} &  K^*(Z\times W\times Q)  \\
K^*(W) \ar[u]_{p_W^*} \ar[r]^{\beta^*} & K^*(W)  \ar[u]_{p_W^*}  }
\end{displaymath}
where $p_W: Z\times W \times Q \rightarrow W$ is the projection map and $\beta$ was the original finite order homeomorphism on $W$. Since the vertical map is an isomorphism, we have that $(\tilde{\zeta})^* = \beta^*$. The same result holds for the induced maps on cohomology.
\end{proposition}
\begin{proof}
By Theorem \ref{homotopyLemmaNonMfldCase}, the minimal homeomorphism $\tilde{\beta} : S^d\times W \times Q \rightarrow S^d\times W \times Q$ is homotopic to $\id_{S^d} \times \beta \times \id_Q$. Using this fact, the induced maps on $K$-theory fit into the following commutative diagram:
\begin{displaymath} 
\xymatrix{ 
K^*(Z \times W \times Q) \ar[r]^{\tilde{\zeta}^*} &  K^*(Z\times W\times Q) \\
K^*(S^d\times W\times Q) \ar[u]_{q^*}  \ar[r]^{\tilde{\beta}^*} & K^*(S^d\times W\times Q)  \ar[u]_{q^*}  \\
K^*(W) \ar[u]_{p^*} \ar[r]^{\beta^*} & K^*(W)  \ar[u]_{p^*}  }
\end{displaymath}
where $q$ is the factor map $Z\times W\times Q \rightarrow S^d\times W\times Q$ and $p: S^d \times W \times Q \rightarrow W$ is the projection map. The result follows by noting that $p_W: Z\times W\times Q \rightarrow W$ is equal to $p \circ q$. The proof that the same result holds in cohomology is similar and is therefore omitted.

\end{proof}
We summarize the results of this section in the following theorem.
\begin{theorem} \label{SummaryTheorem}
Suppose that $W$ is a connected finite CW-complex and $\beta: W \rightarrow W$ is a finite order homeomorphism. Then there exist a connected compact metric space $X$ and minimal homeomorphism $\alpha: X\rightarrow X$ such that
\begin{enumerate}
\item $K^*(X) \cong K^*(W)$ and $\alpha^* = \beta^*$ on $K$-theory, 
\item $H^*(X) \cong H^*(W)$ and $\alpha^* = \beta^*$ on \v{C}ech cohomology.
\end{enumerate}
\end{theorem}
\begin{proof}
The relevant space is $X=Z\times W\times Q$ and the relevant homeomorphism is $\alpha=\tilde{\zeta}$. (The definition of $\tilde{\zeta}$ can be found just before the statement of Proposition \ref{FacMapConTwoInGeneral}.) That this homeomorphism is minimal follows from Proposition \ref{MinMapConTwoInGeneral} and that it has the required induced map on $K$-theory and \v{C}ech cohomology follows from Proposition \ref{SameMapOnKtheory}.
\end{proof}
\begin{remark} \label{manifoldCase}
In the previous proof, note that if the connected finite CW-complex $W$ is a closed connected manifold, then we do not need to take the Cartesian product with the Hilbert cube. In this case one can use the space $Z\times W$. 
\end{remark}

\section{Induced automorphism on $K$-theory and cohomology}

\subsection{Examples starting from a manifold}
In what follows $(Z, \zeta)$ is the dynamical system from Theorem~\ref{ThmAboutZ} and we will apply the construction in the previous section to a number of specific examples. 

\begin{example}
Let $S^n$ be the $n$-sphere and consider the space $Z\times S^n$. There are two possible induced maps on the $K$-theory of the $n$-sphere, and both can be realized by finite order homeomorphisms. It follows from Theorem \ref{SummaryTheorem} and Remark \ref{manifoldCase} that both these induced maps can be realized by minimal homeomorphisms on $Z\times S^n$.
\end{example}

\begin{example}
Consider $Z\times \R^n/\Z^n$ and $B$ an $n$ by $n$ matrix with integer entries that satisfies the following conditions:
\[ {\rm det}(B)=\pm 1 \hbox{ and }B^L=I \hbox{ for some }L\ge 1. \] 
We define a map on $\R^n/\Z^n$ via $[v] \mapsto [Bv]$ where $v\in \R^n$ and $[v]$ denotes the associated element in $\R^n/\Z^n$. The condition ${\rm det}(B)=\pm 1$ implies that this map is a homeomorphism (in fact a diffeomorphism) and the condition $B^L=I$ implies that it is finite order. 

Theorem \ref{SummaryTheorem} and Remark \ref{manifoldCase} imply that there is a minimal homeomorphism on $Z \times \R^n/\Z^n$ with induced map on $H^1(Z \times \R^n/\Z^n) \cong \Z^n$ given by $B$.
\end{example}

\begin{example}
This example is based on a question of Nielsen \cite{MR13306}. Suppose that $W$ is a finite CW-complex and $\varphi: W \rightarrow W$ is a homeomorphism such that for some $K\ge 1$, $\varphi^K$ is homotopic to the identity. Then one can ask if there exists a homeomorphism $\beta: W\rightarrow W$ such that $\beta$ has finite order and is homotopic to $\varphi$. In general this is not possible, see \cite[Introduction]{MR2365406}. However for orientable surfaces Nielsen proved that this is always the case \cite{MR13306}. 

In our context, Nielsen's theorem and our results imply the following:  Suppose $N$ is a closed manifold admitting a free $S^1$-action, $M_g$ is an orientable surface of genus $g$ and $\varphi: M_g \rightarrow M_g$ is a homeomorphism such that for some $K\ge 1$, $\varphi^K$ is homotopic to the identity. Then $N \times M_g$ admits a minimal homeomorphism that is homotopic to $\id_N \times \varphi$. To see that this is the case, Nielsen's theorem implies that there is a homeomorphism $\beta : M_g \rightarrow M_g$ that has finite order and is homotopic to $\varphi$. We can then apply Theorem \ref{MainTheoremNonManifoldCase} to $N \times M_g$ with the finite order homeomorphism $\beta$ to obtain a minimal homeomorphism that is homotopic to $\id_N \times \beta$. This completes the proof since $\id_N \times \beta$ is homotopic to $\id_N \times \varphi$. 

Likewise, using Theorem \ref{SummaryTheorem} and Remark \ref{manifoldCase}, $Z\times M_g$ admits a minimal homeomorphism that has action on $K$-theory given by $\varphi^*$, where $(Z, \zeta)$ is the dynamical system from Theorem~\ref{ThmAboutZ}. 
\end{example}

\subsection{Examples starting from finite CW-complexes}

The following question is a natural question about finite CW-complexes but Theorem \ref{MainTheoremNonManifoldCase} and Proposition \ref{SameMapOnKtheory} give it added importance in our context (see Theorem \ref{RealizingBgeneral} below): \vspace{0.2cm} \\
{\bf A realization question in $K$-theory:} Given a finitely generated abelian group $G$ and a finite order automorphism $\sigma: G \rightarrow G$, does there exist a pointed connected finite CW-complex $W$ such that $\tilde{K}^0(W)\cong G$ and $K^1(W)\cong \{ 0\}$ and based point preserving finite order homeomorphism $\beta_W: W \rightarrow W$ such that $\beta_W^*=\sigma$? \vspace{0.2cm} \\
It would be natural to ask that the orders of $\beta_W$ and $\sigma$ are equal, but since our results do not require this, we do not ask this here. The  realization question is related to a classical question of Steenrod, see \cite[Introduction]{MR244215}. Steenrod's question was about singular homology (rather than $K$-theory) and Swan provided the first counterexamples in \cite{MR244215}. 

In this section we discuss some examples where the realization question in $K$-theory has a positive answer. These particular cases are strong enough to give the $C^*$-algebraic applications considered in \cite{DPS:minDSKth}. Moreover, we hope they illustrate some of the issues and proof techniques arising in the study of such problems. The interested reader can see more details on these techniques in the case of Steenrod's question in, for example, Section 2 of \cite{MR431150}. 

\begin{theorem} \label{RealizingBgeneral}
Suppose $G$ is a finitely generated abelian group and $\sigma$ is a finite order automorphism for which the answer to the ``realization question in $K$-theory" is yes. Then there exists a metric space $X$ and minimal homeomorphism $\tilde{\beta}$ such that $\tilde{K}^0(X)\cong G$, $K^1(X) = \{0\}$, and $\tilde{\beta}^*=\sigma$. Moreover, we can take the minimal homeomorphism to be uniquely ergodic. 
\end{theorem}
\begin{proof}
Since the realization question has a positive answer we can take $W$ a pointed connected finite CW-complex and $\beta_W: W \rightarrow W$ a based point preserving finite order homeomorphism such that $\beta_W^*=\sigma$. Then we can apply the results in Subsection~\ref{PointLikeSpacesWithBeta} (see in particular Proposition \ref{SameMapOnKtheory}) to $X=Z\times W\times Q$ with the finite order homeomorphism $\beta= \beta_W \times \id_Q$ to get the required minimal homeomorphism $\tilde{\zeta}$, which can be taken to be uniquely ergodic by Theorem \ref{MainTheoremNonManifoldCase}.
\end{proof}

\begin{theorem} \label{InductionProcess}
Suppose that $G_1, \ldots, G_k$ are finitely generated abelian groups with finite order automorphisms $\sigma_i: G_i \rightarrow G_i$, $i=1, \ldots, k$. If the realization question in $K$-theory has a positive answer for each $(G_i, \sigma_i)$, then the realization question in $K$-theory for 
\[ (G_1 \oplus \ldots \oplus G_k, \sigma_1 \oplus \ldots \oplus \sigma_k) \]
also has a positive answer.
\end{theorem}
\begin{proof}
Induction reduces the proof to the $k=2$ case. Let $(W_1, \beta_1)$ and $(W_2, \beta_2)$ denote solutions to the realization problem in $K$-theory for $(G_1, \sigma_1)$ and $(G_2, \sigma_2)$, respectively. Since we are working with pointed spaces and based point preserving maps, we can form $(W_1 \vee W_2, \beta_1 \vee \beta_2)$. One can check that $\beta_1 \vee \beta_2$ has finite order. The wedge axiom in reduced $K$-theory implies the result holds.
\end{proof}

\begin{example}
In this example we consider a special case of Lemma \ref{matrixLemmaCW} below. It should aid the reader when considering the more general situation of Lemma \ref{matrixLemmaCW}.
Let
\[
B=\left[ \begin{array}{ccc}0 & 0 & -1 \\ 1 & 0 & -1 \\ 0 & 1 & -1 \end{array}\right].
\]
Note that $B^4=I$ (and $B^n\neq I$ for $1\le n \le 3$), so $B$ gives a $\Z/4\Z$-group action on $\Z^3$. 

Our goal is the construction of a connected finite CW-complex $W$ and a based point preserving homeomorphism $\beta: W\rightarrow W$ such that $\tilde{K}^0(W) \cong \Z^3$, and the action of $\beta$ on reduced $K$-theory is given by $B$. To this end, consider the circle $S^1$ with the trivial $\Z/4\Z$ action and the wedge of four copies of $S^1$, denoted by $S^1 \vee S^1 \vee S^1 \vee S^1$, with the $\Z/4\Z$ action generated by permuting the copies of $S^1$ cyclically. We use the wedge point as the based point and denote an element in $S^1 \vee S^1 \vee S^1 \vee S^1$ by $(x, i)$ where $x\in S^1$ and $i=1,2,3,4$ indicates the circle $x$ is an element in. At the level of the $K^1$-groups we have
\[
K^1(S^1) \cong \Z,
\]  
 with the trivial $\Z/4\Z${-action}, and
\[
K^1(S^1 \vee S^1 \vee S^1 \vee S^1) \cong \Z^4,
\]
 with the $\Z/4\Z$-action generated by 
 \[A=\left[ \begin{array}{cccc}0 & 0 & 0 & 1 \\ 1 & 0 & 0 & 0 \\ 0 & 1 & 0 & 0 \\ 0 & 0 & 1 & 0 \end{array}\right].\]
There is an equivariant map $f: S^1 \vee S^1 \vee S^1 \vee S^1 \rightarrow S^1$ defined by $(x, i) \mapsto x$. Moreover the induced map on the $K_1$-groups 
\[ f^*: K^1(S^1)\cong \Z \rightarrow K^1(S^1 \vee S^1 \vee S^1 \vee S^1) \cong \Z^4 \] 
is given by
\[
[n] \mapsto \left[ \begin{array}{c} n \\ n \\ n \\ n \end{array}\right] .
\]
Let $W=C_f$ be the reduced mapping cone of $f$ (where we note that $C_f$ is a connected finite CW-complex and is pointed because we have taken the reduced mapping cone). Using the fact that $\tilde{K}^0(S^1 \vee S^1 \vee S^1 \vee S^1) \cong \tilde{K}^0(S^1) \cong \{ 0 \}$ and long exact sequence in reduced $K$-theory, we have the following:
\[
0 \rightarrow K^1(C_f) \rightarrow \Z \rightarrow \Z^4 \rightarrow \tilde{K}^0(C_f) \rightarrow 0.
\]
It follows that $\tilde{K}^0(C_f) \cong {\rm coker}{f^*} \cong \Z^3$ and $K^1(C_f) \cong {\rm ker}{f^*}\cong \{ 0\}$ as abelian groups. However, there is more structure involved. Since the mapping cone construction is functorial and $f$ is equivariant with respect to the $\Z/4\Z$-actions, there exists $\beta: C_f \rightarrow C_f$ a based point preserving homeomorphism of order four making the previous exact sequence of abelian groups into an exact sequence of abelian groups with actions of $\Z/4\Z$. 

Now let us show that $\beta^*: \tilde{K}^0(C_f) \rightarrow \tilde{K}^0(C_f)$ is given by $B$. To do so, let $e_1$, $e_2$ and $e_3$ be the images of 
\[
\left[ \begin{array}{c} 1 \\ 0 \\ 0 \\ 0 \end{array}\right], \ \left[ \begin{array}{c} 0 \\ 1 \\ 0 \\ 0 \end{array}\right] \hbox{ and }\left[ \begin{array}{c} 0 \\ 0 \\ 1 \\ 0 \end{array}\right] .
\]
in ${\rm coker}(f^*)$ respectively. This gives an explicit realization of the isomorphism $\tilde{K}^0(C_f)\cong {\rm coker}(f^*)\cong \Z^3$. Moreover, equivariance implies that $\beta^*(e_1)=e_2$ since 
\[
\left[ \begin{array}{cccc}0 & 0 & 0 & 1 \\ 1 & 0 & 0 & 0 \\ 0 & 1 & 0 & 0 \\ 0 & 0 & 1 & 0 \end{array}\right] \left[ \begin{array}{c} 1 \\ 0 \\ 0 \\ 0 \end{array}\right]= \left[ \begin{array}{c} 0 \\ 1 \\ 0 \\ 0 \end{array}\right],
\]
and likewise implies that $\beta^*(e_2)=e_3$. Finally, $\beta^*(e_3)=-e_1-e_2-e_3$ since
\[
\left[ \begin{array}{cccc}0 & 0 & 0 & 1 \\ 1 & 0 & 0 & 0 \\ 0 & 1 & 0 & 0 \\ 0 & 0 & 1 & 0 \end{array}\right] \left[ \begin{array}{c} 0 \\ 0 \\ 1 \\ 0 \end{array}\right]= \left[ \begin{array}{c} 0 \\ 0 \\ 0 \\ 1 \end{array}\right],
\]
and
\[
\left[ \begin{array}{c} 0 \\ 0 \\ 0 \\ 1 \end{array}\right]=\left[ \begin{array}{c} -1 \\ -1 \\ -1 \\ 0 \end{array}\right] ,
\]
as elements in ${\rm coker}(f^*)$.
Thus $\beta^*$ is given by
\[
B=\left[ \begin{array}{ccc}0 & 0 & -1 \\ 1 & 0 & -1 \\ 0 & 1 & -1 \end{array}\right],
\]
as we have computed its action with respect to $e_1$, $e_2$, and $e_3$. This completes the $K$-theory part of the example.

We now consider cohomology. Suppose that $n_0$ is a positive, even integer. Consider the $n_0-1$ dimensional sphere $S^{n_0-1}$ with the trivial $\Z/4\Z$ action and the wedge of four copies of $S^{n_0-1}$, denoted by $S^{n_0-1} \vee S^{n_0-1} \vee S^{n_0-1} \vee S^{n_0-1}$, with the $\Z/4\Z$ action generated by permuting the copies of $S^{n_0-1}$ cyclically. We use the wedge point as the based point and denote an element in $S^{n_0-1} \vee S^{n_0-1} \vee S^{n_0-1} \vee S^{n_0-1}$ by $(x, i)$ where $x\in S^{n_0-1}$ and $i=1,2,3,4$ indicates the circle $x$ is an element in. The reduced cohomology of these spaces vanishes except for the follows:
\[
\tilde{H}^{n_0-1}(S^{n_0-1}) \cong H^{n_0-1}(S^{n_0-1}) \cong K^1(S^{n_0-1}) \cong \Z,
\]  
 with the trivial $\Z/4\Z${-action}, and
\begin{align*}
\tilde{H}^{n_0-1}(S^{n_0-1} \vee S^{n_0-1} \vee S^{n_0-1} \vee S^{n_0-1}) & \cong H^{n_0-1}(S^{n_0-1} \vee S^{n_0-1} \vee S^{n_0-1} \vee S^{n_0-1}) \\
& \cong K^1(S^{n_0-1} \vee S^{n_0-1} \vee S^{n_0-1} \vee S^{n_0-1})  \\
& \cong \Z^4,
\end{align*}
 with the $\Z/4\Z$-action generated by 
 \[A=\left[ \begin{array}{cccc}0 & 0 & 0 & 1 \\ 1 & 0 & 0 & 0 \\ 0 & 1 & 0 & 0 \\ 0 & 0 & 1 & 0 \end{array}\right].\]
Let $W=C_f$ be the reduced mapping cone of $f$ (where we note that $C_f$ is a connected finite CW-complex and is pointed because we have taken the reduced mapping cone). The non-trivial part of the long exact sequence in reduced cohomology has the same form as the one for reduced $K$-theory:
\[
0 \rightarrow H^{n_0-1}(C_f) \rightarrow \Z \rightarrow \Z^4 \rightarrow H^{n_0}(C_f) \rightarrow 0.
\]
where the map $\Z \rightarrow \Z^4$ is given by
\[
[n] \mapsto \left[ \begin{array}{c} n \\ n \\ n \\ n \end{array}\right] .
\]
It follows that $\tilde{K}^0(C_f) \cong H^{n_0}(C_f)\cong \Z^3$. The revelant finite order homeomorphism is the same one discussed above in the case of $K$-theory. We omit the details.
\end{example}

\begin{lemma} \label{matrixLemmaCW}
Consider the $d$ by $d$ matrix
\[
B=\left[ \begin{array}{cccccc}0 & 0 & 0 & \cdots & 0 & -1 \\ 1 & 0 & 0 & \cdots & 0 & -1 \\ 0 & 1 & 0 & \cdots & 0 & -1 \\ 0 & 0 & 1 & \cdots & 0 & -1 \\ \vdots & \vdots & \vdots & \ddots & \vdots & \vdots \\ 0 & 0 & 0 & \cdots & 1 & -1 \end{array}\right]
\]
as an automorphism of $\Z^d$ and let $n_0$ be an even positive integer. Then there exist a connected finite CW-complex $W$ and finite order homeomorphism $\beta$ such that $H^{n_0}(W)\cong \tilde{K}^0(W)\cong \Z^d$, all other non-zero degree cohomology and $K$-theory groups trivial, and such that $\beta^*$ is given by $B$ on both reduced cohomology and reduced $K$-theory.
\end{lemma}
The proof of the previous lemma is the same as the construction in the previous example, but with general $d \in \mathbb{N} \setminus \{0\}$ not necessarily equal to $3$. The details are therefore omitted. However, it is worth mentioning that the space $W$ is the reduced mapping cone associated to the map $f: S^{n_0-1} \vee S^{n_0-1} \vee \ldots  \vee S^{n_0-1} \rightarrow S^{n_0-1}$ defined by $(x, i) \mapsto x$.

\begin{theorem} \label{RealizingB}
Consider the $d$ by $d$ matrix
\[
B=\left[ \begin{array}{cccccc}0 & 0 & 0 & \cdots & 0 & -1 \\ 1 & 0 & 0 & \cdots & 0 & -1 \\ 0 & 1 & 0 & \cdots & 0 & -1 \\ 0 & 0 & 1 & \cdots & 0 & -1 \\ \vdots & \vdots & \vdots & \ddots & \vdots & \vdots \\ 0 & 0 & 0 & \cdots & 1 & -1 \end{array}\right]
\]
as an automorphism of $\Z^d$ and let $n_0$ be an even positive integer. Then there exists a connected metric space $X$ and minimal homeomorphism $\tilde{\beta}$ such that $H^{n_0}(X)\cong \tilde{K}^0(X)\cong \Z^d$ with the other reduced cohomology and reduced $K$-theory groups trivial such that $\tilde{\beta}^*$ is given by $B$ on both reduced cohomology and reduced $K$-theory.
\end{theorem}
\begin{proof}
Lemma \ref{matrixLemmaCW} implies that the answer to the realization question in $K$-theory is ``yes" in this special case. The statement about $K$-theory then follows from Theorem \ref{RealizingBgeneral} and the statement about cohomology can be obtained from the fact that the $K$-theory and cohomology of the CW-complex constructed in this situation are isomorphic; the details are omitted. 
\end{proof}

An important consequence of the previous theorem is that we have many minimal homeomorphisms that are not homotopic to the identity. In addition, the minimal homeomorphisms constructed in the previous theorem play an important role in the $C^*$-algebraic applications considered in \cite{DPS:minDSKth}. In particular, the following result will be used in \cite{DPS:minDSKth}. The condition that $H^1(X)$ is trivial has been included so that the crossed products in \cite{DPS:minDSKth} have no non-trivial projections, see \cite[Corollary 4.6]{DPS:minDSKth}.

\begin{theorem}Suppose that $G_0$ and $G_1$ are finitely generated abelian groups and the realization problem in $K$-theory for $\sigma_0: G_0 \rightarrow G_0$ and $\sigma_1: G_1 \rightarrow G_1$ each have a positive answer. Then there exist a uniquely ergodic minimal homeomorphism on a compact metric space X such that $H^1(X)$ is trivial, $K^0(X)\cong \Z \oplus G_0, K^1(X) \cong G_1$ and the induced maps on $K$-theory are $\id_{\Z} \oplus \sigma_0$ and $\sigma_1$.
\end{theorem}
\begin{proof}
Let $W_0$ be the finite CW-complex and $\beta_0: W_0\rightarrow W_0$ be the finite order homeomorphism that solve the realization problem in $K$-theory for $\sigma_0$, and likewise let $W_1$ be the finite CW-complex and $\beta_1: W_1\rightarrow W_1$ the finite order homeomorphism that solve the realization problem in $K$-theory for $\sigma_1$. 

Take $W=W_0 \vee \Sigma W_1$ where $\Sigma W_1$ is the reduced suspension of $W_1$, and let $\beta_W = \beta_0 \vee \Sigma \beta_1$. We note that $\beta_W$ is well defined because the spaces and maps considered are pointed and the wedge product is taken at this point. Since $\beta_0$ and $\beta_1$ have finite order, $\beta_W$ also has finite order.

The space $W$ satisfies
\[
K^0(W) \cong \Z \oplus G_0 \hbox{ and } K^1(W) \cong G_1,
\]
and the induced map on $K$-theory from $\beta_W$ is
\[
\beta_W^*= \id_{\Z} \oplus \sigma_0 : \Z \oplus G_0 \rightarrow \Z \oplus G_0 \hbox{ and } \beta^*=\sigma_1 : G_1 \rightarrow G_1.
\]
Finally consider $\Sigma^2 W$ and $\Sigma^2 \beta_W$ (that is, take the reduced suspension twice). Bott periodicity implies that $\Sigma^2 W$ and $\Sigma^2 \beta$ have the same $K$-theory properties as $W$ and $\beta_W$. Moreover, for $k\ge 1$, 
\[ H^k(\Sigma^2 W) \cong H^{k-2}(W). \]
Hence $H^1(\Sigma^2 W) \cong H^{-1}(W)$ is trivial. We can apply the results in Subsection~\ref{PointLikeSpacesWithBeta} (see in particular Proposition \ref{SameMapOnKtheory}) to $X=Z\times \Sigma^2 W\times Q$ with the finite order homeomorphism $\beta=\id_Z \times \Sigma^2 \beta_W \times \id_Q$ to get the required minimal homeomorphism. By Theorem \ref{MainTheoremNonManifoldCase} we may take the homeomorphism to be uniquely ergodic. Finally, we note that $H^1(X) \cong H^1(\Sigma^2 W)$ is trivial, as required.
\end{proof}

\end{document}